\documentclass[a4paper,reqno,11pt,oneside]{amsart}
\usepackage{amssymb,amsmath,amsthm,amsfonts}
\usepackage[sc]{mathpazo}
\usepackage{fix-cm}
\usepackage[utf8]{inputenc} 
\usepackage[T1]{fontenc} 
\usepackage[english]{babel} 
\usepackage{nicefrac} 
\usepackage{hyperref}
\usepackage{geometry}
\usepackage{graphicx}
\usepackage{braket}
\usepackage{subfig}
\usepackage{csquotes}
\usepackage{color,epsfig}
\usepackage[nobysame]{amsrefs}

\providecommand{\abs}[1]{\left|#1\right|}
\providecommand{\norm}[1]{\left \| #1\right \|}

\newtheorem{proposition}{Proposition}[section]
\newtheorem{theorem}[proposition]{Theorem}

\newtheorem{lemma}[proposition]{Lemma}

\newtheorem{remark}[proposition]{Remark}

\numberwithin{equation}{section}

\newcommand{\R}{\mathbb{R}}

\newcommand{\Rcal}{\mathcal R}

\newcommand{\eps}{\varepsilon}

\newcommand{\dys}{\displaystyle}
\title[Symmetric Positive solutions to nonlinear Choquard Equations with potentials]{Symmetric Positive solutions to nonlinear  Choquard Equations 
with potentials }
\date{}
\author{Liliane Maia}
\address[L. Maia]{Departamento de Matem\'atica, 
Universidade de Bras\'\i lia, 70910-900 Bras\'\i lia, Brazil.}
\email[L. Maia]{lilimaia@unb.br}
\author{Benedetta Pellacci}
\address[B. Pellacci]{Dipartimento di Matematica e Fisica,
Universit\`a della Campania  ``Luigi Vanvitelli'',  via A. Lincoln 5, 81100
Caserta, Italy.}
\email[B. Pellacci]{benedetta.pellacci@unicampania.it}
\author{Delia Schiera}
\address[D. Schiera]{Dipartimento di Matematica e Fisica,
Universit\`a della Campania  ``Luigi Vanvitelli'',  via A. Lincoln 5, 81100
Caserta, Italy.} 
\email[D. Schiera]{delia.schiera@unicampania.it}
\thanks{Research partially supported by:   PRIN-2017-JPCAPN Grant: ``Equazioni 
differenziali alle derivate parziali non lineari'',
by project Vain-Hopes within the program VALERE: VAnviteLli pEr la RicErca and by the INdAM-GNAMPA group. 
L.\ Maia was partially supported by FAPDF, CAPES, and CNPq grant 309866/2020-0.}

\begin{document}

\begin{abstract}Existence results for a class of Choquard equations with 
potentials  are established.  
The potential has  a limit at infinity and it is taken invariant under the action of a closed subgroup of linear isometries of $\mathbb{R}^N$. As a consequence,
the positive solution found  will be invariant under the same action. 
Power nonlinearities with exponent  greater or equal than two or less than two will be handled. Our results  include the physical case. 
\end{abstract}

\maketitle

\section{Introduction}
This paper is devoted to the study of existence results for the  following Choquard equation
\begin{equation}\label{Choquardeq}\tag{$P_V$}
\begin{cases}
-\Delta u + V(x) u =  (I_{\alpha} \ast \left| u \right| ^{p}) \left| u \right| ^{p-2} u ,
& x\in \mathbb{R}^N, 
\\
\hskip2.1cm u\in H^{1}(\R^{N}), &
\end{cases}\end{equation}
with  $N\geq 2 $, $ \alpha<N$,   and $I_{\alpha}$ represents the Riesz potential of order $\alpha$, defined for every point $x \in \mathbb{R}^N \setminus \{ 0 \}$ by
\[ 
I_{\alpha}(x)=\frac{A_{\alpha}}{\abs{x}^{N-\alpha}}, \quad \text{where } \quad A_{\alpha}=\frac{\Gamma(\frac{N-\alpha}{2})}{\Gamma(\alpha/2) 2^{\alpha} \pi^{N/2}},
\]
where $\Gamma$ denotes the Gamma function (see \cite{Riesz}).
The exponent $p$ lies in the range
\begin{equation}\label{range}
\frac{N-2}{N+\alpha} < \frac1p < \frac{N}{N+\alpha},
 \end{equation}
and  the potential $V$ satisfies
\begin{equation}\label{V1} 
V \in C^0(\R^{N}), \quad \inf_{x \in\mathbb{R}^N} V(x) >0, \quad\text{ and  } \lim_{\abs{x} \to \infty} V(x) =V_{\infty}>0.
\end{equation}
 This partial differential equation arises in several physical models;   it has been introduced 
in \cite{Pekar} (see also \cite{Lieb}) in the context of quantum mechanics,  and it also 
corresponds to the stationary case associated to the nonlinear Hartree equation  (see e.g. \cite{Lions} and for further references see \cite{MorozVanSchaftJFPTA}).

The condition on the exponent $p$  and the  Hardy-Littlewood-Sobolev inequality
implies that the right hand side on \eqref{Choquardeq} is well defined  for every $u\in H^{1}(\R^{N})$,  so that under \eqref{V1} any solution turns out to be a critical point of the action functional
\begin{equation}\label{eq:defI}
\mathcal{I}_V(u)=\frac{1}{2} \int_{\mathbb{R}^N}(\abs{\nabla u}^2 + V(x)u^2) - \frac{1}{2p} \int_{\mathbb{R}^N} (I_{\alpha} \ast \abs{u}^p) \abs{u}^p . 
\end{equation}
When $V(x)\leq V_{\infty},\, V\not\equiv V_{\infty}$, the existence of a least action 
solution, corresponding to a  minimum point of  $\mathcal{I}_V$ on 
\[
{\mathcal N}_{V}:\left\{u\in H^{1}(\R^{N})\setminus \{0\} : \langle \mathcal{I}'_V(u),u\rangle=0\right\}
\]
has been first proved in \cite{Lions} by means of the well-known concentration compactness method (see also \cite{MorozVanSchaftJFPTA, MorozVanSchaftJDE, VanShaftXia}).

But, when $V(x)$ approaches $V_{\infty}$ from above, or oscillating, 
the search of a minimum point on ${\mathcal N}_{V}$ is useless and   higher action level
solutions have to be sought;  this immediately requires a deep study of the  possibile
lack of compactness of a bounded Palais-Smale sequence due to the unboundedness of the 
domain.
 In this path,  versions of the well-known Splitting Lemma, firstly introduced by
\cite{struwe},  are of great help as compactness results, as
this  tool furnishes compactness at quantizied energy intervals 
whenever the so-called problem at infinity 
\begin{equation}\label{Choqlimit}\tag{$P_{\infty}$}
\begin{cases}
-\Delta u + V_{\infty} u =  (I_{\alpha} \ast \left| u \right| ^{p}) \left| u \right| ^{p-2} u & \text{in } \mathbb{R}^N, 
\\
\hskip1.8cm u\in H^{1}(\R^{N}), &
\end{cases}
\end{equation}
 has a unique positive solution. As a consequence, a good knowledge 
concerning existence and uniqueness properties of \eqref{Choqlimit}
turns out to be a cornerstone of the research  of bound states of \eqref{Choquardeq}.
The  existence of  positive solutions to this autonomous problem
goes back to \cite{Lieb} and \cite{Lions}, and  these results have been extended in  \cite{MorozVanSchaftJFA}, where it is shown 
that the problem \eqref{Choqlimit} has a positive radially symmetric least action solution $\omega \in C^2(\mathbb{R}^N)$ for any $p$ satisfying \eqref{range}.

Besides, the uniqueness property  is known for \eqref{Choqlimit} if $p=2$, $\alpha=2$ and $N=3, 4, 5$, (see \cite{Lieb, MaZhao, wangyi} and  \cite{Xiang} for a generalization). Consequently,  in this range of the parameters one can exploit the above  approach to get solutions for \eqref{Choquardeq} and we have followed this strategy in  \cite{MaPeScProc} (see also \cite{WangQuXiao} where a non-autonomous case has been studied for $N=3$ and $\alpha=p=2$).

On the other hand, one can exploit symmetric properties of the potentials $V$ in order to
look for critical points of ${\mathcal I}_{V}$ enjoying the same kind of symmetry,
by minimizing ${\mathcal I}_{V}$ on a symmetric ${\mathcal N}_{V}$.

The introduction  of symmetry into play naturally increases the least action level
and, at the same time, allows to construct a Palais-Smale sequence which is also a 
minimizing one.
Then, the key point becomes again proving that a ``symmetric''  minimizing level lies 
in a range where compactness properties hold. In order to do this, one can build a 
suitable  competitor making use of $\omega$, a least action solution of  
\eqref{Choqlimit}.
More precisely,  consider  the action of  $G$ a closed subgroup of linear isometries of $\R^N$ and set
\[ 
\ell(G)=\min \left\{ \# Gx : \, x \in \mathbb{R}^N \setminus \{ 0 \} \right\}
 \]
where $\# Gx$ is the cardinality of the $G$-orbit of $x$, 
(see   \eqref{def ell} in Section \ref{setting}). Then, 
a good competitor is made of  a sum of suitable translation of 
$\omega$  centred in points that are  far away from each others. 
So that, in order to estimate the functional, decay properties of $\omega$ are required.

When studying the nonlinear Schr\"odinger equation following this strategy, the exponential decay of the least action solution of the corresponding problem at 
infinity plays a crucial role (see e.g. \cite{BahriLi});
since, due to this fast  decay one can  split  the action level of this competitor 
into a suitable  multiple of the least action level of  \eqref{Choqlimit}.

In the case of \eqref{Choqlimit} the decay of  $\omega$ is  of exponential type 
when $p\geq 2$, and existence results
of  symmetric solutions have been proved in \cite{clasal} under suitable hypotheses on the group of  symmetries $G$.

But, when $p<2$ one sees the real non-locality feature of \eqref{Choqlimit} as the 
nonlinearity in this case is not, roughly speaking,  ``focusing enough'' and the decay
of $\omega$ is not exponential any more but it is actually of polynomial type (see \cite{MorozVanSchaftJFPTA,  MorozVanSchaftJDE} for a comprehensive discussion on this topic).

Our  contributions in this paper are twofold. 
First of all, we succeed in proving existence  results in the case $p<2$ by 
performing the above mentioned asymptotic analysis  even when $\omega$ 
decays polynomially and under assumptions on the group of symmetries $G$ weaker than the one in \cite{clasal}.

In particular, our existence result for $p<2$ is the following one.
\begin{theorem}\label{mainthm1}
Let $G$ be a closed subgroup of the linear isometries of $\R^{N}$, with  $\ell(G) \ge 2$, and finite.
Assume that the exponent $p$ is such that
\begin{equation}\label{pless2}
\frac{N+\alpha}N<p<2
\end{equation}
and that the potential $V(x)$ satisfies  \eqref{V1} and 
\begin{equation}\label{Vless2}
V(x) \le V_{\infty} + A_0 (1+\abs{x})^{-\beta}, \quad \forall  x \in \R^N, \text{ with } A_0 >0, \text{ and } \beta > \frac{N-\alpha}{2-p}. 
\end{equation}
Then, if  $V$ is $G-$invariant,  Problem \eqref{Choquardeq} has a 
positive $G-$invariant solution.
\end{theorem}
In proving this result  different phenomena, compared to the Schr\"odinger 
equation, appear due to the fact that $p<2$. 
First of all, ${\mathcal N}_{V}$ is not of class $C^{1}$; however, the approach 
introduced in \cite{SzulkinWeth}  can be exploited to deal with this situation.
More relevant are the difficulties arising when studying the term involving the 
nonlinearity. Indeed, classical algebraic inequalities 
as in \cite{BahriLi, clasal} do not apply and the 
 effects of the nonlinearity are spread in a less concentrated region,  so that interactions  between 
two different translations of $\omega$ become relevant even if the centres of translation are
far away from each other (see Lemma \ref{stime epsilon precise}). 
This phenomenon, due to the polynomial decay of $\omega$,  resembles to what 
occur  in other context involving nonlocal operators. 

Once one has dealt with the nonlinearity,
then the integral term involving the potentials has to be compared with
the analogous term in \eqref{Choqlimit}; in this comparison, decay estimates such
as \eqref{Vless2} are useful and, coherently with the decay of $\omega$, we
assume here a polynomial type  decay on the potential $V$; this appears to be
another novelty compared with the nonlinear Schr\"odinger equation and it is 
closer to what happen in the zero mass case (see \cite{ClappMaia2,ClappMaiaPellacci}).

Existence results of symmetric solutions  for \eqref{Choquardeq},
also in the presence of a magnetic potential, have been obtained in \cite{CingolaniClappSecchi}   for  $p$ satisfying \eqref{range} when 
$\ell(G)=\infty$,  or for $p\geq 2$ and satisfying \eqref{range} and for 
potentials  $V(x)$ going to $V_{\infty}$ from below.
While, the case of $\ell(G)$ finite, and $V(x)$ approaching to $V_{\infty}$ from above or oscillating has been tackled in \cite{clasal} for $p\geq 2$ and $\ell(G)\geq 3$. 

So this is, to our knowledge, the first existence result of a positive solution 
enjoying a finite  number of symmetries when $p<2$.

In addition, we succeed in obtaining a refined asymptotic  analysis also for
$p\geq 2$. This enables us to weaken the request on the decay of $V$ and on
the group of symmetry  allowing $\ell(G)\geq 2$.

Our results depending on the range where the exponent $\alpha$ lies
are the following.

\begin{theorem}\label{mainthm2}
Let $G$ be a closed subgroup of the linear isometries of $\R^{N}$, with  
$\ell(G) \ge 2$ and finite.
Assume \eqref{V1} and 
\begin{equation}\label{p>2}
\frac{N-2}{N+\alpha}< \frac1p<\frac12, \quad \text{ or } \quad p=2, \; \alpha\le N-1.
\end{equation}
Let $\mu_{G}$ be defined in \eqref{defmug} and  suppose that, for every $x\in\R^{N}$  it holds
\begin{equation}\label{VgeqN-1}
V(x) \le V_{\infty} + A_0 (1+\abs{x})^{\sigma}e^{-\beta\abs{x}} , \quad \text{with } A_0 >0,\; \beta\geq \mu_{G}\sqrt{V_{\infty}}
\end{equation}
with the exponent $\sigma$ satisfying  the following condition depending on the 
constant  $\mu_G$ 
\begin{equation}\label{eq:sigma1}
\begin{cases}
\sigma\in \R & \text{if }  \beta > \mu_G \sqrt{ V_{\infty}}
\\
\sigma < \min \left\{- 1, -\frac{N-1}{2} +2\tau_1 \right\} &\text{if } 
\beta = \mu_G \sqrt{ V_{\infty}}, 
\end{cases}\end{equation}
and
\[
 \tau_1= \begin{cases}
0 & \text{ if } p>2 \text{ or } p=2, \, \alpha < N-1\\
{\nu}\frac{\sqrt{V_{\infty}}}2 & \text{ if } p=2, \, \alpha = N-1,
\end{cases}\]
with $\nu>0$  given in \eqref{Q}.

Then, if  $V$ is $G$-invariant,  Problem \eqref{Choquardeq} has a 
positive $G$-invariant solution. 
\end{theorem}
Theorem \ref{mainthm2} does not cover the case $p=2$ and $\alpha\in (N-1,N)$.
This is because for $\alpha $ lying in this range the decay of least action 
solutions of \eqref{Choqlimit} has a perturbation in the exponential term
(see \eqref{decaypertexp})  so that, hypothesis \eqref{VgeqN-1} is not suitable
any more, and we will prove the following result.
\begin{theorem}\label{mainthm4}
Let $G$ be a closed subgroup of linear isometries of $\mathbb{R}^N$, with 
$\ell(G) \ge 2$ and finite. Assume \eqref{V1} and 
\begin{equation}\label{palpha}
p=2,\quad   \alpha\in \left(N-1,N-\frac12\right],\quad
\text{and let }\;
\gamma=1-N+\alpha \in \left(0, \frac12\right].
\end{equation}
Let $\mu_{G}$ be defined in \eqref{defmug} and  suppose that, for every $x\in\R^{N}$  it holds
\begin{equation}\label{VgreatN-1} 
V(x) \le V_{\infty} + A_0(1+ \abs{x})^\sigma e^{-\beta \abs{x}+ c'\abs{x}^{\gamma'}}, \;  
\quad \text{with } A_0 >0,\; \beta\geq \mu_{G}\sqrt{V_{\infty}}
\end{equation}
where (recalling  \eqref{decaypertexp},  and \eqref{Q}) the constants 
 $\gamma' \in [0, 1)$, $c' \ge 0$, $\sigma \in \R$ are such that
\begin{itemize}
\item[(i)] If $\beta > \mu_G \sqrt{V_\infty}$, then $\gamma' \in [0, 1)$, $c' \ge 0$, $\sigma \in \R$. 
\item[(ii)] If $\beta= \mu_G \sqrt{V_\infty}$, 
we assume $\gamma'\leq \gamma$ and  
\[
 c' \ge 0, \quad \sigma \in \R, \qquad \text{ if }  \gamma'<\gamma.
\]
\item[(iii)] If  $\beta= \mu_G \sqrt{V_\infty}$ and $\gamma'=\gamma$ we assume that 
$\mu_{G}<2$, $c' \leq 2^{1-\gamma} c_\gamma  \mu_G^\gamma$ and $\sigma $ is  such that
\begin{equation*}
\begin{cases}
\begin{split}
\, \sigma &\in \R,\hskip3.2cm  \text{ if } 
c' < 2^{1-\gamma} c_\gamma \mu_G^\gamma, \,  \\
\, \sigma &< -\frac{N-1}{2} +\frac \gamma 2 + 2\tau_2 \quad
\text{ if } 
c' = 2^{1-\gamma} c_\gamma \mu_G^\gamma, 
\quad\text{with 
}
\tau_2 =
\begin{cases} 0 &\text{ if } \alpha < N-\frac12 \\
\frac{\sqrt{V_\infty} \nu}{8} &\text{ if } \alpha = N-\frac12. \end{cases}\end{split} 
\end{cases}\]
\end{itemize}
Then if  $V$ is $G$-invariant,  Problem \eqref{Choquardeq} has a positive $G$-invariant solution. 
\end{theorem}
Theorems \ref{mainthm2} and \ref{mainthm4} extend Theorem 1.3 in \cite{clasal} 
under various aspects. First of all, as already mentioned, we also include the case 
$\ell(G)=2$. 
Moreover, even in the case $\ell(G) \ge 3$, we improve the decay assumptions on 
$V$, and we also study the threshold case $\beta=\mu_G \sqrt{V_\infty}$.

As mentioned above, when $p>2$ the decay of $\omega$ is 
analogous to the one of the unique positive solution of the nonlinear 
autonomous  Schr\"odinger  equation (see \cite{BerLions}), so that,   we 
are naturally  lead to  assume that the potential $V(x)$ approach its limit at infinity 
decaying in an exponential way too. Moreover,  differently  to the case of Theorem 
\ref{mainthm1},  here we see the effect of the symmetry in the decay due to the 
presence of the constant $\mu_{G}$, which takes into account the least distance 
between two 
elements of the $G$-orbit of a point $x$ in the unit sphere.
This marks another relevant difference with the case $p<2$ as $\mu_{G}$ does not play any role in this latter case.

When the exponent $\beta$ in the decay  of $V(x)$ reaches the  threshold
$\mu_{G}\sqrt{V_{\infty}}$, the corrections become important. In particular,
 in Theorem \ref{mainthm2} this role is played by the polynomial part and
 we have to take into account that
when $p=2$  and $\alpha=N-1$ a polynomial perturbation with exponent
$\tau_{1} $ (see \eqref{decay2}) appears  in the decay of $\omega$, 
so that we are
 naturally lead to assume  \eqref{eq:sigma1}. 
 
In the setting of Theorem \ref{mainthm4}, we first have to 
consider the fact that when $\alpha$ overcome $N-1$, the decay of $\omega$ 
changes again and an exponential perturbation arises (see \eqref{decaypertexp});  
in this situation one has to face new difficulties, which can be overcome  by 
means of some new  technical lemma (see Lemma \ref{lemma estensione1} 
\ref{lemma estensione2}) which we believe that may 
be of  independent interest.
In addition, in this case the threshold is achieved when both $\beta=\mu_{G}\sqrt{V_{\infty}}$ and
 $\gamma'=\gamma$; at this point we need the condition $\mu_{G}<2$ and 
 we see that the constant $c'$ becomes relevant: if $c'<2^{1-\gamma} c_\gamma  \mu_G^\gamma$ ($c_{\gamma}$ given in \eqref{decaypertexp}) then the exponential term still guide the asymptotic, while
 if $c'=2^{1-\gamma} c_\gamma  \mu_G^\gamma$ again the polynomial part starts being the leading term and we arrive at the condition on $\sigma$.

When $p=2$, $\alpha > N- \frac{1}{2}$, new perturbations appear in the 
decay of $\omega$ (see \cite{MorozVanSchaftJFA}) and we
expect that  similar results could be obtained, by slightly modifying our arguments 
and at the price of some heavy technicalities. The first step would be to prove an 
extension of  Lemma \ref{lemma estensione1}, and \ref{lemma estensione2}  to the case of functions with more involved exponential corrections. 

Concluding, let us point out that our 
conditions on the potential $V$  are somewhat sharp, meaning  that, 
if they are not satisfied the decay of  $V$ may be not comparable with the 
asymptotic decay of the  solutions of the limit problem, see also Remark 
\ref{rmk:condizioniottimali}. 

 \medskip
 
This paper is organized as follows: in Section \ref{setting} we give the variational setting of the problem and some preliminary results, whereas in Section \ref{sec:p<2}  we study the case $p<2$, and prove Theorem \ref{mainthm1}. Theorem \ref{mainthm2} and Theorem \ref{mainthm4} will be proved in Section \ref{sec:pgeq2} through a unified approach. 
 
\section{Setting of the problem and preliminaries}\label{setting}

In this section we introduce the symmetric framework in which we settle
our problem. Let us observe that the use of symmetry turns out to be a useful
and largely exploited tool when looking for existence results to 
\eqref{Choquardeq} (see \cite{Ackermann, clasal1, DAveniaMederskiPomponio, GhimentiVanSchaft}).
 
In what follows, $G$ will represent a closed subgroup of linear isometries of $\mathbb{R}^N$. Define the $G$-orbit of $x$ as $Gx=\{ gx: \, g \in G \}$, and 
$\#G x$ its cardinality. We set 
\begin{equation}\label{def ell} 
\ell(G)=\min \left\{ \# Gx : \, x \in \mathbb{R}^N \setminus \{ 0 \} \right\}. \end{equation}
As  mentioned in the Introduction, the case $\ell(G) = +\infty$
has been treated in \cite[Theorem 1.1]{CingolaniClappSecchi}.
Here, we will assume 
\[
\ell(G) < +\infty. 
\]
\begin{remark}\label{re:lG}
In general, there may exist points such that $\#G x > \ell(G)$. For instance, 
take in $\mathbb{R}^4\cong \mathbb{C} \times \mathbb{C}$ the group $G=\mathbb{Z}_2 \times \mathbb{Z}_3 $, where $\mathbb{Z}_l$ is the cyclic group generated by the $l$-th roots of the unit. Then the point $x=(0,0,0,1)\cong(0,i)$ has $\#G x= 3$, and $\ell(G)=2$ as it possible to
see taking $y=(1,0,0,0)\cong(1,0)$.
\end{remark}
In Section \ref{sec:p<2} we will just use the notion of $\ell(G)$, while 
in Section \ref{sec:pgeq2} the minimum distance between two different 
orbit points will play a role in the exponential decay estimates.

More precisely, let us consider the set $\Sigma$ given by
\begin{equation} \label{def:sigma}
\Sigma=\left\{ x \in S^{N-1}: \, \#G x=\ell(G) \right\}.
\end{equation}
Let us define 
\begin{equation} \label{eq:muGz}
\mu(Gz)=\begin{cases}
\inf \{ \abs{ gz -h z}: g, h \in G, gz \ne hz \}, & \text{ if } \#G z \ge 2 \\
2\abs{z} & \text{ if } \#G z=1,
\end{cases}
\end{equation}
for every $z\in \Sigma$, and  the extremum  
\begin{equation}\label{defmug}
\mu_G=\inf_{z \in \Sigma} \mu(Gz).
\end{equation}
The following properties of $\Sigma $ and $\mu_{G}$ will be 
useful in Section \ref{sec:pgeq2}.
\begin{lemma}\label{le:sigmamu}
The set $\Sigma\neq \emptyset$ is a  compact, G-invariant subset of $\R^{N}$ and
  $\mu_{G}$ is achieved.
\end{lemma}
\begin{proof}
The set $\Sigma $ is nonempty, because $\ell(G)$ is attained, and 
for every $x \in \mathbb{R}^N \setminus \{ 0 \} $ such that $ \#G x=\ell(G)$, then
$x/\abs{x} \in S^{N-1}$ and  $\#G x=\#G (x/\abs{x})$, since elements in $G$ are linear transformations, so that $x/\abs{x} \in \Sigma$.  
In addition, the G-invariance property is a direct consequence of the definition.

In order to prove that $\Sigma$ is closed,  let $(x_n)\in\Sigma$ be such that 
$x_n\to x$. Arguing by contradiction, assume that $x\notin\Sigma$. 
Then the orbit of $x$ contains a number of  points greater than $\ell(G)$, so that
there exist $g_1,\ldots,g_{\ell+1}\in G$ with $g_ix\neq g_jx$ for every $i\neq j$, 
$i,j=1,\ldots,\ell+1$. As $g_ix_n\to g_ix$ for every $i$, we get that $g_i x_n\neq 
g_j x_n$ if $i\neq j$, for $n$ large enough, i.e., $\#Gx_n\geq \ell+1$, 
which cannot be as  $x_n\in\Sigma$. 
This shows that $\Sigma$ is closed, and as it is contained in $S^{N-1}$, it turns 
out  to be a compact set.

Let us now define the function $f:\Sigma \mapsto \R$  as
\[
f(x):=\mu(Gx), \qquad \text{where $\mu_{G}$ is defined in \eqref{eq:muGz}},
\]
and prove that $f$ is continuous. 
Let $x_n\to x$ in $\Sigma$. Choose $g_1,\ldots,g_\ell\in G$ such that $Gx=\{g_1x,\ldots,g_\ell x\}$. 
 Arguing as before, one obtains that $g_{i}x_{n}\neq g_{j}x_{n}$ for $i,\,j=1,\dots \ell(G)$ and for $n$ sufficiently large, so that 
 $G(x_{n})=\{g_1x_{n},\ldots,g_\ell x_{n}\}$, because $x_{n}\in \Sigma$.
Then, the continuity immediately follows if $\# Gx=1$, otherwise, it
results
$$
f(x_n)=\min_{i\neq j}|g_ix_n-g_jx_n|\to \min_{i\neq j}|g_ix-g_jx|=f(x)\qquad\text{as \ }n\to\infty.
$$
So $f$ is continuous, as claimed.
As a consequence, $\mu_{G}$  is achieved.
\end{proof}
The influence of symmetries will appear in the decay estimates 
for $p\geq 2$ through the constant $\mu_{G}$.
In the  following  remarks we  give some information  on $\mu_{G}$ that 
 will be useful  in Section \ref{sec:pgeq2} and that illustrate some hypotheses
 of Theorem \ref{mainthm4}.
\begin{remark}\label{re:muG}
Notice that $0<\mu_{G}\leq  2.$ Moreover, if  $\mu_{G}=2$ then
$\ell(G)=1$ or $\ell(G)=2$.

Indeed, the first inequality is a direct consequence
of the fact that $\mu_{G}$ is attained.
On the other hand, the second inequality follows by observing that the distance between  two distinct points on the unit  sphere is less or equal than two.

Furthermore, suppose by contradiction that  $\mu_{G}=2$ and $\ell(G)\geq 3$.
Then, there exists $x\in \Sigma$ such that $\mu(Gx)=2$ and there exist
$g_{1},\,g_{2},\,g_{3}\in G$ such that $g_{i}x\neq g_{j}x$.
Without loss of generality, we can assume that 
$|g_{1}x-g_{2}x|=\mu(Gx)=\mu_{G}=2$, but then $|g_{1}x-g_{3}x|<2$ as $|g_{i}x|=1$ for 
every $i=1,\,2,\,3$, which contradicts the fact that the minimum  is $\mu_{G}=2$.

 \end{remark}
\begin{remark}
In conclusion (iii) of Theorem \ref{mainthm4} we assume $\mu_{G}<2$.
Notice that one can find groups such that $\mu_G < 2$ and $\ell(G)=2$. For instance, let $g$ the linear isometry in $\R^3$ which corresponds to a clockwise rotation of angle $\pi/2$ around the $y$ axis, followed by a clockwise rotation of angle $\pi$ around the $z$ axis. Take the closed group acting on $\R^3$ generated by $g$. Then $\ell(G)\ge2$, as every point on the sphere is mapped by $g$ in a point different from itself. Moreover, consider the north pole $N=(0,0, 1)$. This point is mapped into $(1, 0, 0)$ by $g$ and $g^{-1}$, and $g^2(N)=N = (g^{-1})^2(N)$, thus $\#G N=2$, $\ell(G)=2$ and $N \in \Sigma$. However, the distance between $N$ and $(1, 0, 0)$ is less than 2, thus $\mu_G <2$.
\end{remark}

As observed in the introduction, our results cover the case 
$\ell(G)=2$. Note that there are 
many groups satisfying $\ell(G) \ge 3$ when $N=2n$ is even, but this is not the case if $N$ is odd. For further remarks concerning $\ell(G)$ see 
\cite[pg.4]{clasal}.

We will work in the functional space 
\[
H^1_G= \left\{ u \in H^1(\R^{N}): \, u(gx)=u(x) \, \text{ for any } g \in G, \, x \in \mathbb{R}^N \right\}
\]
 endowed, thanks to \eqref{V1}, with the scalar product and norm
\[ 
( u, v )_V = \int_{\mathbb{R}^N} (\nabla u \nabla v + V(x) uv), \quad \quad \norm{u}^2_V=\int_{\mathbb{R}^N}(\abs{\nabla u}^2 + V(x) u^2 ). 
\]
Every symmetric solution to \eqref{Choquardeq} is a critical point of the action
functional $\mathcal{I}_V: H^{1}_{G}\mapsto \R$ defined in \eqref{eq:defI}.
Indeed, $\mathcal{I}_V(u)$ is G-invariant as $V$ is, so that 
the  principle of symmetric criticality (\cite{pal}) applies.

Hypothesis \eqref{range} and Hardy-Littlewood-Sobolev inequality
imply that $\mathcal{I}_V$ is a $C^{1}$ functional on $H_G^1$, (see \cite[Proposition 3.1]{MorozVanSchaftJFPTA}), so that we can define
\[
\langle \mathcal{I}'_V(u), u \rangle=\norm{u}_V^2 - \int_{\mathbb{R}^N} (I_{\alpha} \ast \abs{u}^p) \abs{u}^{p} \]
and
\begin{equation}\label{defnehari}
\mathcal{N}_V^G=\left\{ u \in H^1_G({\mathbb{R}^N}) \setminus \{ 0 \} : 
\langle \mathcal{I}'_V(u), u \rangle=0\right\},\qquad 
 c_V^G=\inf_{u \in \mathcal{N}_V^G} \mathcal{I}_V(u).
 \end{equation}
Notice that twice differentiability of $\mathcal{I}_V$ holds only for $p \ge 2$, 
(see for instance \cite{MorozVanSchaftJFA}). 
As a consequence, $\mathcal{N}_V^G$ is not, in general, a differentiable 
manifold.
In order to overcome these difficulties we will use the approach in \cite{SzulkinWeth} (see Section \ref{proofs1}).

In an analogous way, let us define $\mathcal{I}_{\infty}: H^{1}(\R^{N})\mapsto \R$ by
\[
\mathcal{I}_{\infty}(u)=\frac{1}{2} \int_{\mathbb{R}^N}(\abs{\nabla u}^2 + V_{\infty}u^2) - \frac{1}{2p} \int_{\mathbb{R}^N} (I_{\alpha} \ast \abs{u}^p) \abs{u}^p ,
\]
where  $H^{1}(\R^{N})$ is endowed with the scalar product and the norm
\begin{equation}\label{normainfty}
( u, v ) = \int_{\mathbb{R}^N} (\nabla u \nabla v + V_{\infty} uv), \quad \quad \norm{u}^2=\int_{\mathbb{R}^N}(\abs{\nabla u}^2 +V_{\infty}  u^2 )
\end{equation}
and accordingly $\mathcal{N}_\infty^G$ and 
$c_{\infty}^G$ are defined for \eqref{Choqlimit}. 

The existence of a least action solution to \eqref{Choqlimit} is proved, under 
assumption \eqref{range}, in  Theorem 3.2 in \cite{MorozVanSchaftJFPTA}. 
Moreover, weak solutions are classical, and, up to translation and inversion of the 
sign, positive and radially symmetric, see \cite{MorozVanSchaftJFA, Lieb}. 
Precise decay asymptotic  for solutions to \eqref{Choqlimit} are 
given in Propositions 6.3, 6.5 and Remark 6.1 in \cite{MorozVanSchaftJFA}, 
(see  also \cite{MorozVanSchaftJDE}),  depending on the value of $p$. 

Summarizing the following result holds.
\begin{theorem}[Theorem 4 pg.157 in \cite{MorozVanSchaftJFA}]
Assume $\alpha\in (0,N)$ and that $p$ satisfies \eqref{range}.
Let $\omega$ be a least action solution to \eqref{Choqlimit}.
Then the following asymptotic estimates hold.
\begin{enumerate}
\item
If $p<2$, there exists a positive constant $c$ such that  
\begin{equation}\label{decay}
\omega(x)=(c+o(1))\abs{x}^{-\frac{N-\alpha}{2-p}}
  \quad 
\text{ as } \abs{x} \to \infty. \end{equation}
\item
Under \eqref{p>2}, it results
\begin{equation}\label{decay2} 
\omega(x)= (c+o(1))\abs{x}^{-\frac{N-1}{2}+\tau_{1}}
e^{-\sqrt{ V_{\infty}}\abs{x}} 
\quad \text{ as } \abs{x} \to \infty. 
\end{equation} 
where $\tau_{1}=0$ if $p>2$ or $p=2$ and $\alpha<N-1$; while $\tau_{1}= 
\frac{\sqrt{V_{\infty}} \nu}{2}$ when $p=2$ and $\alpha=N-1$ and 
where $\nu$ is a positive constant depending on the $L^{2}(\R^{N})$ norm of $\omega$ (see \eqref{Q} below).
\item
If $p=2$ and $N-1<\alpha\leq N-\frac12$,
then $\omega$ decays as follows
\begin{equation}\label{decaypertexp}  
\omega(x) = (c+o(1)) \abs{x}^{-\frac{N-1}{2}+\tau_{2}}e^{-\sqrt{V_{\infty}}\abs{x}+c_\gamma \abs{x}^{\gamma}}, \qquad \text{with $\gamma=1-N+\alpha$, }
\end{equation} 
and where $c_{\gamma}=\frac 1 \gamma \nu^{1-\gamma} \sqrt{V_\infty}$; 
$\tau_{2}=0$ if $\alpha<N-\frac12$ and $\tau_{2}= \frac{\sqrt{V_\infty} \nu}{8}$
when $\alpha=N-1/2$, and $\nu$ is as in \eqref{Q} below.
\end{enumerate}
\end{theorem}
The above result shows that the decay of the least action solutions strongly 
depends on the interaction of  the Riesz potential and on the nonlinearity.
First of all, when $p>2$ we see the same decay behavior as in the local
case (see \cite{BerLions}); while for $p<2$ the presence of the convolution term
forces the decay to be of polynomial type, more resembling the case of nonlocal
fractional operators (see \cite{FQT} and \cite{MorozVanSchaftJFA} p. 157-158). 
 
The threshold is $p=2$. As observed in \cite{MaPeScProc}, in this range we see different perturbations on the decay of $\omega$ depending on $\alpha$. 
In general, it holds
\begin{equation}\label{decay3} 
\omega(x)= (c+o(1)) \frac{e^{- \sqrt{V_{\infty}}Q(\abs{x})}}{\abs{x}^{\frac{N-1}{2}}} 
\quad \text{ as } \abs{x} \to \infty, 
\end{equation}
where 
\begin{equation}\label{Q} Q(t)= \int_{\nu}^{t} \sqrt{1 - \frac{\nu^{N-\alpha}}{s^{N-\alpha}}}  \, ds, 
\qquad  \nu^{N-\alpha} = \frac{1}{V_{\infty}} \frac{\Gamma(\frac{N-\alpha}{2})}{\Gamma(\frac{\alpha}{2}) \pi^{N/2} 2^{\alpha} } \int_{\mathbb{R}^N} \abs{\omega}^2, 
\end{equation}
and $\nu$ only depends on $\norm{\omega}_2^2$  (see \cite{MorozVanSchaftJFA}).
Then, taking into account the Taylor expansion of the square root, one can see 
that  \eqref{decay2} still holds when $\alpha < N-1$; while a 
perturbation  in the polynomial part occurs  if $\alpha=N-1$ (which includes the physical case $N=3$, $\alpha=2$, $p=2$),  and more and more 
perturbations appear  as $\alpha $ increases. 
   In particular, if $N-1 < \alpha \le N- \frac{1}{2}$, the decay becomes as stated
in \eqref{decaypertexp}.
As a last information, when $\alpha >N-\frac12$ the decay will include more and more terms in the Taylor expansion of the function $Q$
 (for more details see also Remark 6.1 in \cite{MorozVanSchaftJDE}).

In order to obtain analogous decay estimates on the convolution term
the following lemma will be crucial
\begin{lemma}\label{le:convo}
 Let $h \ge 0$, $h\in L^{\infty}$ such that 
\begin{equation}\label{bound}
\sup_{\mathbb{R}^N} h(x) (1+\abs{x})^{s} < + \infty, \qquad \text{for some $s >N$.} \end{equation}
 Then
\begin{equation}\label{eq:conv1} 
I_{\alpha} \ast h(x) = I_{\alpha}(x) \norm{h}_1 (1+ o(1)). 
\end{equation}
Moreover, let $f \in L^p_{\text{loc}}(\R^{N})$, $f \ge 0$, be such that 
\[
\dys \sup_{\R^{N}} f(x)(1+\abs{x})^{\eta} <+ \infty,\quad \text{with $p \eta >N$.}
\]
For every  $z_1,\,z_2\in \mathbb{R}^N$, it results
\begin{equation}\label{eq:conv2}
\limsup_{\abs{x} \to \infty} \abs{x-z_1}^{(N-\alpha)\frac{p-1}{p}} \abs{x-z_2}^{(N-\alpha)\frac{1}{p}} 
\int_{\mathbb{R}^N} \frac{f(y-z_1)^{p-1} f(y-z_2)}{\abs{y-x}^{N-\alpha }} \, dy < + \infty. \end{equation}
\end{lemma}
\begin{proof}
The first conclusion follows immediately from Lemma 6.2 in \cite{MorozVanSchaftJFA}.
In order to prove the second  one, we observe that 
\[
\abs{y-x}^{N-\alpha}=\abs{y-x}^{\frac{p-1}{p}(N-\alpha)} \abs{y-x}^{(N-\alpha)\frac{1}{p}}.
\]
Therefore, by applying H{\"o}lder's inequality, one has
\begin{align*}
\int_{\mathbb{R}^N} \frac{f(y-z_1)^{p-1} f(y-z_2)}{\abs{y-x}^{N-\alpha}} dy
&\le
		\left(\int \frac{f(y-z_1)^{p}}{\abs{y-x}^{N-\alpha}} \right)^{\frac{p-1}{p}} 
		\left(\int \frac{{f(y-z_2)}^{p}}
		{\abs{y-x}^{N-\alpha}} 
		\right)^{\frac{1}{p}}\\
		&=
		\left[\int \frac{f(y)^{p}dy}{\abs{y+z_{1}-x}^{N-\alpha}} \right]^{\frac{p-1}{p}} 
		\left[\int \frac{{f(y)}^{p}dy}
		{\abs{y+z_{2}-x}^{N-\alpha}} 
		\right]^{\frac{1}{p}}
		\\
		&=
		\left[\left( I_\alpha \ast f^p\right)(x-z_{1})\right]^{\frac{p-1}{p}}
		\left[\left( I_\alpha \ast f^{p}\right)(x-z_2)\right]^{\frac{1}{p}}.
\end{align*}
The conclusion follows by applying \eqref{eq:conv1}  with $h=f^p$, and $s=p \eta>N$. 
\end{proof}
As an immediate consequence of Lemma \ref{le:convo}, we get the following asymptotic  decay of the convolution term  
\begin{equation}\label{decayconv}
I_{\alpha} \ast \omega^p(x) = I_{\alpha}(x) \norm{\omega}^{p}_p (1+ o(1)). \end{equation}
Indeed, if $p<2$, \eqref{decay} yields that \eqref{bound} is satisfied by $h=\omega^p$   
with $s=p \frac{N-\alpha}{2-p}$. Note that $s >N$ as $p > \frac{2N}{2N-\alpha}$, which is always true in our setting, thanks to \eqref{range}.  Moreover, in the case $p \ge 2$, then  \eqref{bound} is satisfied by $h=\omega^p$ for any $s$. 

Let us conclude this section by introducing the threshold that will guide our study. 
Let $\Sigma$ be defined in \eqref{def:sigma}. Then
for every  $z\in \Sigma$ there are $g_1, \dots, g_{\ell(G) }\in G$ 
such that $g_{i}z\neq g_{j}z$ whenever $g_{i}\neq g_{j}$.
We denote 
with  $ \omega_{i, R}(x)$ a solution of the limit problem translated in $Rg_{i}z$,
namely
\begin{equation}\label{omegaR}
 \omega_{i, R}(x)=\omega(x- Rg_i z), \qquad \text{for $i=1,\dots,\ell(G).$}
\end{equation}
Then we define
\begin{equation}\label{defeps}
\begin{split}
 \varepsilon_R^{ij} &= \int_{\mathbb{R}^N} (I_{\alpha} \ast \omega_{i, R}^p) \omega_{i, R}^{p-1} \omega_{j, R}=\int_{{\R^{N}}}\left[\nabla \omega_{i,R}\cdot\nabla \omega_{j,R}+V_{\infty} \omega_{i,R} \omega_{j,R}\right],
\\
\eps_{R}&=\sum_{i\neq j}^{\ell(G)} \eps^{ij}_{R}.
\end{split} \end{equation}
In the following sections we will see that $\eps_{R} $ has different asymptotic  decays  depending on $p$. This will lead us to assume different decay
assumptions on the potentials $V$ in order to get our existence results.
Moreover, in order to prove that $c^{G}_{V}$ given in \eqref{defnehari}
is an action level where the Palais-Smale condition holds, we will evaluate 
${\mathcal I}_{V}$ on the competitor
\begin{equation}\label{defsigma} 
\chi_{R, z}=\sum_{i=1}^{\ell(G)}  \omega_{i, R}, \qquad \text{where 
$\omega_{i, R}$ is defined in \eqref{omegaR}}.
\end{equation}

\section{Case $p<2$}\label{sec:p<2} 

This section is devoted to the proof of Theorem \ref{mainthm1}, which 
will be carried on by minimizing  ${\mathcal I}_{V}$ on ${\mathcal N}^{G}_{V}$.
Since it is known that the  Palais-Smale condition is satisfied for any level below a suitable value  which depends on $\ell(G)$ and on $c_\infty^G$, (see Proposition 3.1 in \cite{CingolaniClappSecchi}),  the main point consists in finding 
a  competitor in ${\mathcal N}^{G}_{V}$ showing that the minimum value
belongs to the  range where compactness holds.

In order to do this evaluation, we first analyze  the decay $\eps_R$ as $R\to \infty$ (see  Lemma  \ref{est epsilon}), and in Lemma \ref{estimatesV} we show that the part involving the potential is actually small with respect to  $\eps_R$. 
Then,  we will analyze the behavior of integrals involving the nonlinearities;  
as explained in the Introduction, the presence of a power $p<2$ prevents one from exploiting algebraic results for power-like nonlinearities, (see \cite{BahriLi, clasal}),  but we will take care of the non-locality feature of the problem
by performing  a very careful analysis of the integrals involved, and on the behavior of $\eps_R$ (see Lemma \ref{stime epsilon precise} and Proposition \ref{prop:somma}). 
Then in subsection \ref{proofs1} we will conclude the proof  also exploiting 
the approach in \cite{SzulkinWeth} as for $p<2$, ${\mathcal N}^{G}_{V}$ is not of class $C^{1}$.

\subsection{Asymptotical analysis }\label{sec:estimates1}
Let us first analyze  the asymptotic decay of $\eps_R$. 
\begin{lemma}\label{est epsilon}
Let $p$ satisfy \eqref{pless2}. Then,  for $R$ large enough
\[
\varepsilon_R^{ij} \sim R^{-\frac{N-\alpha}{2-p}}, \quad \text{where $\varepsilon_R^{ij} $ is introduced in \eqref{defeps}}.
\]
\end{lemma}
\begin{remark}\label{rem:simp<2}
Lemma \ref{est epsilon} shows that for $p<2$ the distance between any two points
in the orbit of $z$ does not play  any role in determining the  decay of $\eps_{R}$.
This marks a relevant difference with the local case and with  the case $p\geq 2$ (see  Lemma \ref{est epsilon2}).
\end{remark}
\begin{remark}\label{dij}
Let $d_{ij}:=\abs{g_i z-g_j z}$. 
Notice that Lemma \ref{le:sigmamu} and Remark \ref{re:muG} imply that $0<\mu_G \le \mu(Gz) \le d_{ij}=\abs{g_i z-g_j z} \le 2$. Indeed, as $g_i$ is an isometry, $\abs{g_i z}=\abs{z}=1$, and distinct points on the sphere have distance $\le2$. 
\end{remark}

\begin{proof}
Let us first observe that, exploiting \eqref{defeps}, \eqref{decay} and \eqref{decayconv}, one has
\begin{align*}
\varepsilon_R^{ij} &\le 
 C \int_{\mathbb{R^{N}}} (1+\abs{x-Rg_i z})^{-\frac{N-\alpha}{2-p}}(1+\abs{x-Rg_j z})^{-\frac{N-\alpha}{2-p}}dx.
\end{align*}
We now apply Lemma \ref{lemma cm} with $a=a'=-\frac{N-\alpha}{2-p}$ and $\xi=Rg_i z-Rg_jz$, and take into account \eqref{range} to get 
$ \varepsilon_R^{ij} \le C R^{- \frac{N-\alpha}{2-p}}. $

In order to get the estimates from below, one takes into consideration that
$\omega$ is positive, radially symmetric and decreasing to  obtain  that 
\begin{align*} \nonumber \inf_{x \in B_1(0)} I_{\alpha} \ast \omega^p(x) &\ge  \inf_{x \in B_1(0)} \frac{A_{\alpha}}{R_0^{N-\alpha}} \int_{B_{R_0}(x)} \omega^p(y) \, dy 
\ge 
A_{\alpha} \frac{\abs{B_{R_0}(0)}}{R_0^{N-\alpha}} \min_{y \in B_{R_0+1}(0)} \omega^p(y) 
\\&\ge C >0. \end{align*}
Hence, again exploiting \eqref{decay}, one has (denoting with $C$ possibly different constants)
\begin{align}\label{stima below eps}
\nonumber \varepsilon_R^{ij} &\ge 
\int_{B_1(Rg_i z)} (I_{\alpha} \ast \omega_{i, R}^p) \omega_{i, R}^{p-1} \omega_{j, R} 
= \int_{B_1(0)} (I_{\alpha} \ast \omega^p(x)) \omega^{p-1}(x) \omega(x- R(g_j z -g_i z)) \, dx \\
\nonumber  &\ge  \inf_{x \in B_1(0)} ( I_{\alpha} \ast \omega^p(x)  \omega^{p-1}(x) ) \int_{B_1(0)} \omega(x- R(g_j z-g_i z)) \, dx \\
& \ge C \int_{B_1(0)} (1+\abs{x-R(g_j z-g_i z)})^{-\frac{N-\alpha}{2-p}} \ge CR^{- \frac{N-\alpha}{2-p}}, 
\end{align}
where the last inequality can be deduced observing that, 
as pointed out in Remark \ref{dij} $d_{ij}\leq 2$  so that, 
if $\abs{x} <1$, the following inequality holds  for every $R \ge 1$
\[
 1+ \abs{x-R(g_j z-g_i z)} < 1 + \abs{x} + R\abs{g_j z-g_i z} < 4R. \qedhere
 \]
\end{proof}
\begin{lemma}\label{estimatesV}
If $p$ satisfies \eqref{pless2} and $V$ satisfies \eqref{Vless2} then it holds
\[ \mathcal{A}_V:=\int_{\mathbb{R}^N} (V(x) - V_{\infty})\left(\chi_{R, z}\right)^2 \le  o(\varepsilon_R). \]
\end{lemma}

\begin{proof}
The conclusion is a direct consequence of  \eqref{decay} and  Lemma \ref{lemma cm} as
\[  \int_{\mathbb{R}^N} (V(x) - V_{\infty}) \omega_{i, R}^2 \le C \int_{\mathbb{R}^N}  (1+\abs{x})^{-\beta}(1+\abs{x-Rg_i z})^{-2\frac{N-\alpha}{2-p}}\le CR^{-\tau} \]
where $\tau =\min \{ \beta, 2\frac{N-\alpha}{2-p}, \beta + 2\frac{N-\alpha}{2-p} - N \} > \frac{N-\alpha}{2-p}$. 
\end{proof}
In order to compare the asymptotic  behavior of the nonlinearity term with respect to $\eps_{R}$, we first need to deepen our knowledge of the behavior of  the threshold  $\eps_R$. 

Having in mind \eqref{defeps}, let us define, for $i,\,j=1,\dots,\ell(G)$
\begin{equation}\label{defeps restr}  
\varepsilon^{ij}_{k l}=\iint_{{\Rcal}_{kl}} 
\frac{ \omega_{i, R}^p(\zeta) \omega_{i, R}^{p-1}(\theta) \omega_{j, R}(\theta)}{|\theta-\zeta|^{N-\alpha}}
d\zeta d\theta 
\end{equation}
where the set ${\Rcal}_{kl}$ is defined,
for $k,\,l=1,\dots,\ell(G)$ and for $\rho \in \left(0, \min_{i, j} d_{ij}/2\right)$ fixed,
as 
\begin{equation}\label{eq:defRkl}
{\Rcal}_{kl}:= 
\left\{ (\theta, \zeta): \abs{\theta-Rg_k z} < \rho R, \;\abs{\zeta-Rg_l z} < \rho R \right\}.
\end{equation}
In the following lemma we detect all the contribution terms in 
$\eps_R$ that actually play a relevant role.
We will see that in this study the presence of the convolution term will be important.
\begin{lemma}\label{stime epsilon precise}
The following expansion holds 
\[ \eps_R = \sum_{i \ne j} (\eps^{ij}_{ij}+\eps^{ij}_{ji}+\eps^{ij}_{ii}) + o(\eps_R). \]
\end{lemma}
\begin{proof}
Taking into account \eqref{defeps}, we need to show that,
for every  $(i, j)$,  $\eps^{ij}_{R}$ restricted to the set 
$\left( \Rcal_{ij} \cup \Rcal_{ji} \cup \Rcal_{ii} \right)^c $
is $o(\eps_R)$.
First notice that 
\[
\left( \Rcal_{ij} \cup \Rcal_{ji} \cup \Rcal_{ii}  \right)^c = \Omega_1 \cup \Omega_2 \cup \Rcal_{jj}, \]
where
\begin{align*}
 \Omega_1 &=(B_i \cup B_j)^c \times \R^N\quad \text{where } B_k =  B_{\rho R}(R g_k z),
 \\
\Omega_2 &= (B_i \cup B_j) \times (B_i \cup B_j)^c.
\end{align*}
In  Figure \ref{fig:epsdiv}  we draw an example for $i=1,\,j=2$, $\ell(G)=2$.
\begin{figure}[ht]
\includegraphics[width=0.65\textwidth]{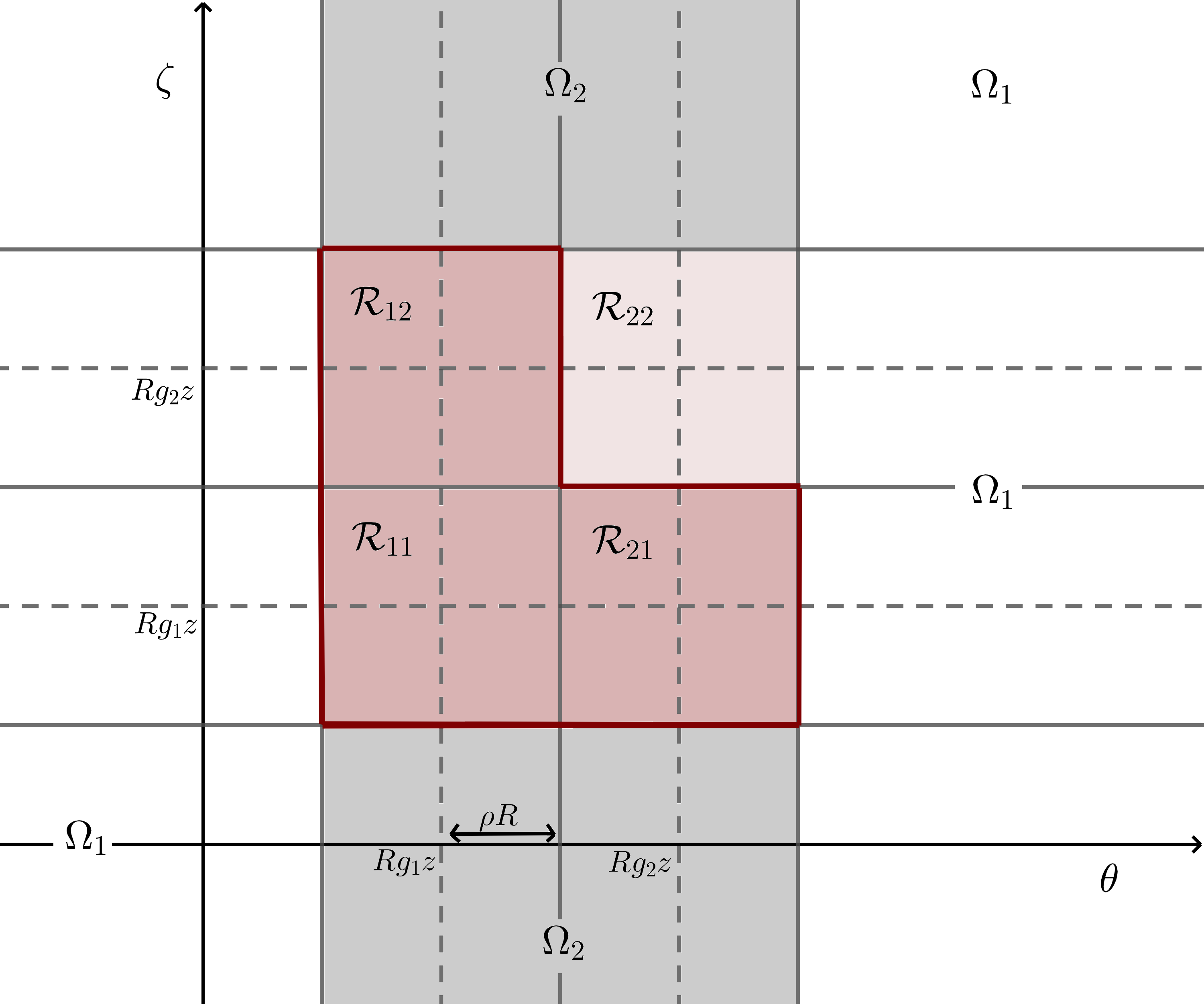}
\caption{\label{fig:epsdiv}}
\end{figure}
Let us start estimating the integral on $\Omega_1$ and 
define  the following subsets of $\R^{N}$
\begin{equation}\label{eq:E+}
 \begin{split}
&E^+=\{ \theta\in \R^{N} : \abs{\theta-R g_j z} > \abs{\theta-R g_i  z} \}, 
\\
& E^-=\{ \theta \in \R^{N} : \abs{\theta-R g_j z} < \abs{\theta-R g_i  z} \}.
\end{split}\end{equation}
Then, exploiting \eqref{decay}, and  \eqref{decayconv}  we have 
\begin{align*}
 \iint_{\Omega_1}  \frac{\omega^p_{i, R}(\zeta) \;
\omega^{p-1}_{i, R}(\theta)\omega_{j, R}(\theta) }{\abs{\theta-\zeta}^{N-\alpha}} \, d\zeta d\theta 
&\le 
 \int_{(B_i \cup B_j)^c}  
( I_\alpha \ast \omega^p_{i, R}(\theta)) \omega^{p-1}_{i, R}(\theta) \omega_{j, R}(\theta)\;d\theta
 \\ 
&   \leq
 C \int_{(B_i \cup B_j)^c} 
\abs{\theta-R g_i z}^{-\frac{N-\alpha}{2-p}}\abs{\theta- R g_j z}^{-\frac{N-\alpha}{2-p}} \,d\theta
\\  &
 \le 
C  \int_{(B_i \cup B_j)^c\cap E^+ } {\abs{\theta-R g_i z}^{-2\frac{N-\alpha}{2-p}}} \,d\theta 
  \\ 
 & \quad \quad + C \int_{(B_i \cup B_j)^c\cap E^-} {\abs{\theta- R g_j z}^{-2\frac{N-\alpha}{2-p}}}  \,d\theta \\
   & \le  C \int_{\rho R}^{+\infty} r^{-2\frac{N-\alpha}{2-p} + N-1} \, dr  
  =  CR^{N-2\frac{N-\alpha}{2-p}} 
 =o(\varepsilon_R),
\end{align*}
where the last inequality can be deduced  noting that $(B_i \cup B_j)^c \cap E^+ \subset (B_i)^c$ (similarly $(B_i \cup B_j)^c\cap E^- \subset (B_j)^c$), and 
applying Lemma \ref{est epsilon}.

Now we estimate the integral on $\Omega_2$.
We have, exchanging integrals, and exploiting  \eqref{decay} and \eqref{eq:conv2}, with $f=\omega$, $z_{1}=Rg_{i}z$ and $z_{2}=Rg_{j}z$  
\begin{align*}
 \iint_{\Omega_2}  \frac{\omega^p_{i, R}(\zeta) \;
\omega^{p-1}_{i, R}(\theta)\omega_{j, R}(\theta) }{\abs{\theta-\zeta}^{N-\alpha}} \, d \zeta d\theta 
&\le \int_{(B_i \cup B_j)^c }\omega^p_{i, R}(\zeta) d \zeta \int_{\R^N} \frac{
\omega^{p-1}_{i, R}(\theta)\omega_{j, R}(\theta) }{\abs{\theta-\zeta}^{N-\alpha}} \, d\theta 
\\
&\le C \int_{(B_i \cup B_j)^c }  
\frac{\omega^{p}_{i, R}(\zeta) }{\abs{\zeta-R g_i z}^{(N-\alpha)\frac{p-1}{p}} \abs{\zeta-R g_j z}^{\frac{N-\alpha}{p}}} \, d\zeta  
\\
&\le C \int_{(B_i \cup B_j)^c }  
\abs{\zeta-Rg_i z}^{-\frac{(N-\alpha)(3p-2)}{p(2-p)}}\abs{\zeta- R g_j z}^{-\frac{N-\alpha}{p}} \, d\zeta .
\end{align*}
Then splitting in $E^{+}$ and in $ E^{-}$ (see \eqref{eq:E+}) and applying
Lemma \ref{est epsilon} one has
\begin{align*}
 \iint_{\Omega_2} & \frac{\omega^p_{i, R}(\zeta) \;
\omega^{p-1}_{i, R}(\theta)\omega_{j, R}(\theta) }{\abs{\theta-\zeta}^{N-\alpha}} \, d \zeta d\theta 
\le
 C \int_{(B_i \cup B_j)^c\cap E^+ } {\abs{\zeta-R g_i z}^{-2\frac{N-\alpha}{2-p}}} \,d\zeta 
  \\ 
&  + C \int_{(B_i \cup B_j)^c\cap E^-} {\abs{\zeta- R g_j z}^{-2\frac{N-\alpha}{2-p}}}  \,d\zeta 
  \le C R^{-2\frac{N-\alpha}{2-p} + N} = o(\eps_R).
\end{align*}
To complete the proof we have to study $\eps_{jj}^{ij}$.  Recall that $\Rcal_{jj}=
B_{j}\times B_{j}$; moreover, \eqref{decay} yields
\[
\begin{split}
 \iint_{\Rcal_{jj}} & \frac{\omega^p_{i, R}(\zeta) \;
\omega^{p-1}_{i, R}(\theta)\omega_{j, R}(\theta) }{\abs{\theta-\zeta}^{N-\alpha}} \,  d\zeta d\theta \leq
\\ 
&\le C  
\int_{B_{j}} \frac{ d\zeta}{(1+\abs{\zeta- R g_i z})^{p\frac{N-\alpha}{2-p}} }  
\int_{B_{j}} \frac{\omega_{j, R}(\theta) \, d\theta }{(1+\abs{\theta- R g_i z})^{(p-1)\frac{N-\alpha}{2-p}}\abs{\theta -\zeta}^{N-\alpha}}. 
\end{split}
\]
In addition, let us observe  that for every  $j \ne i$, and for every $\xi \in \R^{N}$
 it holds
\[
\abs{\xi-Rg_j z} < \rho R \quad \Rightarrow\quad \abs{\xi-Rg_i z} \geq 
\abs{Rg_j z-Rg_i z}-\abs{\xi-Rg_j z}>
\rho R.
\]
Then,
recalling \eqref{decayconv}, one gets
\begin{align*}
 \iint_{\Rcal_{jj}}  \frac{\omega^p_{i, R}(\zeta) \;
\omega^{p-1}_{i, R}(\theta)\omega_{j, R}(\theta) }{\abs{\theta-\zeta}^{N-\alpha}} \, d\zeta d\theta & \le 
\frac{C}{ R^{(2p -1)\frac{N-\alpha}{2-p}}} \int_{B_{j}}  
(I_{\alpha}\ast{\omega_{j, R}})(\zeta)d\zeta
\\
&\le  \frac{C}{ R^{(2p -1)\frac{N-\alpha}{2-p}}}     \int_{B_{j}}\frac{d \zeta}{
 (1+\abs{\zeta- R g_j z})^{N-\alpha}} . 
\end{align*}
Computing the  last integral by using polar coordinates yields (denoting with $C$
possibly different constants)
\begin{align*}
 \iint_{\Rcal_{jj}}  \frac{\omega^p_{i, R}(\zeta) \;
\omega^{p-1}_{i, R}(\theta)\omega_{j, R}(\theta) }{\abs{\theta-\zeta}^{N-\alpha}} \, d\zeta d\theta 
& \le 
C R^{(-2p +1)\frac{N-\alpha}{2-p}}R^\alpha  = o(\eps_R),
\end{align*}
where the last equality follows from  direct computations, taking into account 
Lemma \ref{est epsilon} and \eqref{range}.
\end{proof}
We are now in the position to deal with the nonlinearity term.
\begin{proposition}\label{prop:somma}
Let  $ \chi^p_{R, z}$ be defined in \eqref{defsigma}, 
 $\varepsilon_{R}$  in \eqref{defeps} and $\eps_{kk}^{ki}$  in \eqref{defeps restr}.
It results
\[ \int_{\R^N} \left( I_{\alpha} \ast \chi^p_{R, z}\right) \chi^p_{R, z} \ge  \sum_{k=1}^{\ell(G)}  \int_{\mathbb{R}^N} (I_{\alpha} \ast  \omega_{k, R}^p)  \omega_{k, R}^p 
+ p \eps_R + p \sum_{k=1}^{\ell(G)}\sum_{i \ne k} \eps^{ki}_{kk} + o(\varepsilon_R).  \]
\end{proposition}
\begin{proof}
Recalling \eqref{eq:defRkl}, let us set
\[ \Omega=(\R^N\times \R^N )\setminus  \Big( \bigcup _{k, l=1}^{\ell(G)} {\Rcal}_{kl} \Big),   \]
and observe that from our choice of $\rho$ it follows that the sets ${\Rcal}_{kl}$ are disjoint, so that
\begin{equation}\label{sum}  
\begin{split}
\int_{\mathbb{R}^N} (I_{\alpha} \ast \chi_{R, z} ^p)\chi_{R, z} ^p 
=& 
\sum_{k, l=1}^{\ell(G)}\iint_{{\Rcal}_{kl}} \frac{\chi^p_{R, z}(\zeta) \chi^p_{R, z}(\theta ) }{\abs{\theta-\zeta}^{N-\alpha}}  d\zeta d\theta 
+ \iint_{ \Omega}
\frac{\chi^p_{R, z}(\zeta) \chi^p_{R, z}(\theta ) }{\abs{\theta-\zeta}^{N-\alpha}} d\zeta d\theta 
\\
\geq &
\sum_{k, l=1}^{\ell(G)}\iint_{{\Rcal}_{kl}} \frac{\chi^p_{R, z}(\zeta) 
\chi^p_{R, z}(\theta ) }{\abs{\theta-\zeta}^{N-\alpha}} \, d\zeta d\theta 
\\
=&
\sum_{k=1}^{\ell(G)}\iint_{\Rcal_{kk}} \frac{\chi^p_{R, z}(\zeta) 
\chi^p_{R, z}(\theta ) }{\abs{\theta-\zeta}^{N-\alpha}} \, d\zeta d\theta 
\\&
+
\sum_{k=1}^{\ell(G)}\sum_{l\neq k}^{\ell(G)}\iint_{{\Rcal}_{kl}} \frac{\chi^p_{R, z}(\zeta) 
\chi^p_{R, z}(\theta ) }{\abs{\theta-\zeta}^{N-\alpha}} \, d\zeta d\theta 
\end{split}
\end{equation}
where the  inequality follows from the positivity of the function $\omega$.
Let us  consider the first  integral term on the right hand side. Taking into account \eqref{defsigma} and applying 
Bernoulli's  inequality 
\[
\chi_{R, z}^p =\left(\sum_{i=1}^{\ell(G)} \omega_{i,R} \right)^p \ge \omega_{k,R}^p+p\sum_{i\ne k} \omega_{k,R}^{p-1}\omega_{i,R},
\]
we get 
\begin{equation} \label{C1}
\begin{split}
\iint_{\Rcal_{kk}}
 \frac{\chi^p_{R, z}(\zeta) \chi^p_{R, z}(\theta) }{\abs{\theta-\zeta}^{N-\alpha}}&\, d\zeta d\theta 
 \geq
 \iint_{\Rcal_{kk}} \frac{\omega^p_{k, R}(\zeta) \;
\omega^p_{k, R}(\theta)
}{\abs{\theta-\zeta}^{N-\alpha}} \,d\zeta d\theta
\\  
+p&\sum_{i \ne k}  \iint_{\Rcal_{kk}}
\frac{\omega^p_{k, R}(\zeta) \;
\omega^{p-1}_{k, R}(\theta)\omega_{i, R}(\theta)
}{\abs{\theta-\zeta}^{N-\alpha}} \, d\zeta d\theta
\\ 
+p &\sum_{i \ne k} \iint_{\Rcal_{kk}}
\frac{\omega^{p}_{k, R}(\theta) \; \omega^{p-1}_{k, R}(\zeta)
\omega_{i, R}(\zeta)
}{\abs{\theta-\zeta}^{N-\alpha}} \, d\zeta d\theta
\\ 
+p^2&\sum_{i \ne k}   \iint_{\Rcal_{kk}}
\frac{\omega^{p-1}_{k, R}(\zeta) \omega_{i, R}(\zeta)\;
\omega^{p-1}_{k, R}(\theta)\omega_{i, R}(\theta)
}{\abs{\theta-\zeta}^{N-\alpha}} \, d\zeta d\theta.
\end{split}
\end{equation}
Note that the  last term can be neglected, as it is positive, and  that
the following equality holds (see \eqref{defeps restr})
\[ 
 \iint_{\Rcal_{kk}}
\frac{\omega^{p}_{k, R}(\theta) \; \omega^{p-1}_{k, R}(\zeta)
\omega_{i, R}(\zeta)}{\abs{\theta-\zeta}^{N-\alpha}} \, d\zeta d\theta
= 
\iint_{\Rcal_{kk}}
\frac{\omega^{p}_{k, R}(\zeta) \;
\omega^{p-1}_{k, R}(\theta)\omega_{i, R}(\theta)
}{\abs{\zeta-\theta}^{N-\alpha}} \, d\zeta d\theta= \varepsilon_{kk}^{ki}. \]
Exploiting these facts into \eqref{C1} one obtains
\begin{equation}\label{C11}
\iint_{\Rcal_{kk}}
 \frac{\chi^p_{R, z}(\zeta) \chi^p_{R, z}(\theta) }{\abs{\theta-\zeta}^{N-\alpha}}\, d\zeta d\theta 
 \geq  \iint_{\Rcal_{kk}} \frac{\omega^p_{k, R}(\zeta) \;
\omega^p_{k, R}(\theta)
}{\abs{\theta-\zeta}^{N-\alpha}} \,d\zeta d\theta+ 
2p \, \sum_{i \ne k}  \varepsilon_{kk}^{ki}.
\end{equation}
On the other hand
\[
\begin{split}
\iint_{\Rcal_{kk}} \frac{\omega^p_{k, R}(\zeta) \omega^p_{k, R}(\theta)}{\abs{\theta-\zeta}^{N-\alpha}} \, d\zeta d\theta =&
 \int_{\mathbb{R}^N} (I_{\alpha} \ast \omega^p_{k, R})\omega^p_{k, R} 
-\iint_{(\Rcal_{kk})^c}
\frac{\omega^p_{k, R}(\zeta) \omega^p_{k, R}(\theta) }{\abs{\theta-\zeta}^{N-\alpha}} \, d\zeta d\theta,
\end{split}
\]
so that,  \eqref{C11}  becomes
 \begin{align} \label{eqC12}
 \iint_{\Rcal_{kk}}
 \frac{\chi^p_{R, z}(\zeta) \chi^p_{R, z}(\theta) }{\abs{\theta-\zeta}^{N-\alpha}}\, d\zeta d\theta 
\geq
2p & \sum_{i \ne k}  \varepsilon^{ki}_{kk}+
\int_{\mathbb{R}^N} (I_{\alpha} \ast \omega^p_{k, R})\omega^p_{k, R}(x) \;dx\\
\nonumber&-\iint_{(\Rcal_{kk})^c}
\frac{\omega^p_{k, R}(\zeta) \omega^p_{k, R}(\theta) }{\abs{\theta-\zeta}^{N-\alpha}} 
\, d\zeta d\theta. 
\end{align}
By using  \eqref{decay}, \eqref{decayconv} and recalling \eqref{eq:E+}, we obtain 
\begin{equation}\label{stima kk}
\begin{split}
\iint_{(\Rcal_{kk})^c} \frac{\omega^p_{k, R}(\zeta) \omega^p_{k, R}(\theta)}{\abs{\theta-\zeta}^{N-\alpha}}d\zeta d\theta
\leq &
 \int_{(B_{k})^{c}} \omega^p_{k, R}(\zeta) d\zeta
\int_{\R^{N}}\frac{ \omega^p_{k, R}(\theta)}{\abs{\theta-\zeta}^{N-\alpha}} d\theta
\\
&+ \int_{(B_{k})^{c}} \omega^p_{k, R}(\theta) d\theta
\int_{\R^{N}}\frac{ \omega^p_{k, R}(\zeta)}{\abs{\theta-\zeta}^{N-\alpha}} d\zeta
\\
\leq&
C\int_{(B_{k})^{c}}  \abs{\theta-R g_k z}^{-2\frac{N-\alpha}{2-p}}d\theta
\\
=&
C \int _{\rho R}^{\infty} r^{-2\frac{N-\alpha}{2-p}}r^{N-1}\,dr
=CR^{N-2\frac{N-\alpha}{2-p}} =o(\varepsilon_R),
\end{split}
\end{equation}
where the last equality follows from Lemma \ref{est epsilon} and
hypothesis \eqref{range}.

Using \eqref{stima kk} and \eqref{eqC12} one deduces the following information 
concerning the contributes on $\Rcal_{kk}$
\[
 \begin{split}
 \sum_ {k=1}^{\ell(G)} \iint_{\Rcal_{kk}}
 \frac{\chi^p_{R, z}(\zeta) \chi^p_{R, z}(\theta) }{\abs{\theta-\zeta}^{N-\alpha}}\, d\zeta d\theta 
&\geq \sum_{k=1}^{\ell(G)} \int_{\mathbb{R}^N} (I_{\alpha} \ast \omega^p_{k, R})\omega^p_{k, R}
+2p \sum_{k=1}^{\ell(G)} \sum_{i \ne k} \varepsilon^{ki}_{kk}
 + o(\eps_R). 
\end{split}
\]
Exploiting this in  \eqref{sum} we have
\begin{equation}\label{sum2}  
\begin{split}
\int_{\mathbb{R}^N} (I_{\alpha} \ast \chi_{R, z} ^p)\chi_{R, z} ^p 
&\geq \sum_{k=1}^{\ell(G)} \int_{\mathbb{R}^N} (I_{\alpha} \ast \omega^p_{k, R})\omega^p_{k, R}
+
2p \sum_{k=1}^{\ell(G)} \sum_{i \ne k} \varepsilon^{ki}_{kk}
 + o(\eps_R)\\
&+
\sum_{k=1}^{\ell(G)}\sum_{l\neq k}^{\ell(G)}\iint_{{\Rcal}_{kl}} \frac{\chi^p_{R, z}(\zeta) 
\chi^p_{R, z}(\theta ) }{\abs{\theta-\zeta}^{N-\alpha}} \, d\zeta d\theta.
\end{split}
\end{equation}
Let us now study the integral terms on ${\Rcal}_{kl}$ with $k\neq l$.
By applying Bernoulli's inequality with respect to $\omega_{j, R}$ one has 
\[ 
\chi_{R, z}^p \ge \omega_{j,R}^p+p\sum_{i\neq j}^{\ell(G)}\omega_{j,R}^{p-1}\omega_{i,R} \ge \omega_{j,R}^p+p\omega_{j,R}^{p-1}\omega_{i,R}.
 \]
Computing the product and dropping some terms by positivity, one gets 
\begin{equation}\label{Bern1}
\chi_{R, z}^p(\zeta)  \chi_{R, z}^p(\theta) 
\ge 
p \omega^p_{j, R}(\zeta) \omega^{p-1}_{j, R}(\theta)\omega_{i, R}(\theta) 
+ p \omega^p_{j, R}(\theta) \omega^{p-1}_{j, R}(\zeta)
\omega_{i, R}(\zeta). 
\end{equation}
Since it results
\[ 
\iint_{{\Rcal}_{kl}}
 \frac{\chi^p_{R, z}(\zeta) \chi^p_{R, z}(\theta) }{\abs{\theta-\zeta}^{N-\alpha}}\, d\zeta d\theta = \frac 12 \iint_{{\Rcal}_{kl}}
 \frac{\chi^p_{R, z}(\zeta) \chi^p_{R, z}(\theta) }{\abs{\theta-\zeta}^{N-\alpha}}\, d\zeta d\theta  + \frac 12 \iint_{{\Rcal}_{kl}}
 \frac{\chi^p_{R, z}(\zeta) \chi^p_{R, z}(\theta) }{\abs{\theta-\zeta}^{N-\alpha}}\, d\zeta d\theta
 \]
we can apply \eqref{Bern1} with $j=k$ and $i=l$ in the first integral and with $j=l$ and $i=k$ in the second one, to obtain 
\begin{equation}\label{kl} 
\iint_{{\Rcal}_{kl}}
 \frac{\chi^p_{R, z}(\zeta) \chi^p_{R, z}(\theta) }{\abs{\theta-\zeta}^{N-\alpha}}\, d\zeta d\theta 
 \geq \frac p2 (\eps^{kl}_{lk} + \eps^{kl}_{kl} +  \eps^{lk}_{lk} + \eps^{lk}_{kl}). 
 \end{equation}
Then, recalling \eqref{sum2},  and \eqref{kl}, one has
\begin{align*} 
\int_{\mathbb{R}^N} (I_{\alpha} \ast \chi_{R, z} ^p)\chi_{R, z} ^p \; 
\geq 
&  \sum_{k=1}^{\ell(G)} \int_{\mathbb{R}^N} (I_{\alpha} \ast \omega^p_{k, R})\omega^p_{k, R}(x) \;dx 
+ 2p\sum_{k=1}^{\ell(G)} \sum_{i \ne k} \eps_{kk}^{ki}
\\
&+ \frac p2 \sum_{k=1}^{\ell(G)} \sum_{l \ne k} 
\left(\eps^{kl}_{lk} + \eps^{kl}_{kl} +  \eps^{lk}_{lk} + \eps^{lk}_{kl}\right)+ o(\eps_R) 
\\
=& \sum_{k=1}^{\ell(G)} \int_{\mathbb{R}^N} (I_{\alpha} \ast \omega^p_{k, R})\omega^p_{k, R}  
 + 2p\sum_{k=1}^{\ell(G)} \sum_{i \ne k} \eps_{kk}^{ki}
\\
&
+ p \sum_{k=1}^{\ell(G)} \sum_{l \ne k} \left(\eps^{kl}_{kl} + \eps^{lk}_{kl}\right)+ o(\eps_R). 
\end{align*}
So that,  Lemma \ref{stime epsilon precise} implies
\begin{align*}
\nonumber 2p\sum_{k=1}^{\ell(G)} \sum_{i \ne k} \eps_{kk}^{ki}+p \sum_{k=1}^{\ell(G)} \sum_{l \ne k} (\eps_{kl}^{kl}+\eps_{kl}^{lk}) &= p\sum_{k=1}^{\ell(G)} \sum_{i \ne k} \eps_{kk}^{ki} + p\sum_{k=1}^{\ell(G)}  \sum_{l \ne k} (\eps^{kl}_{kl}+\eps^{lk}_{kl}+\eps^{kl}_{kk}) 
\\
&= p\sum_{k=1}^{\ell(G)} \sum_{i \ne k} \eps_{kk}^{ki} + p\sum_{l\neq k=1}^{\ell(G)}   (\eps^{kl}_{kl}+\eps^{kl}_{lk}+\eps^{kl}_{kk}) 
\\
&= p\sum_{k=1}^{\ell(G)} \sum_{i \ne k} \eps_{kk}^{ki} + p \eps_R + o(\eps_R),
\end{align*}
yielding the conclusion.
\end{proof}

\subsection{Proof of Theorem \ref{mainthm1}}\label{proofs1}

In this subsection we will complete the proof of Theorem \ref{mainthm1}.
Let us start, recalling the useful properties concerning ${\mathcal N}^{G}_{V}$.
\begin{lemma}\label{Nehari homeo}
The following conclusions hold.
\begin{enumerate}
\item
For each  $u \in H^1_G\setminus \{0\}$ there exists a unique $T(u) >0$ such that 
\( T(u) u \in \mathcal{N}_V^G  \) (see \eqref{defnehari}).
Moreover, $T(u)u $ is the unique global maximum of $\mathcal{I}_V(tu)$, $t \in [0, + \infty) $. 
\item
$c_V^G$ defined in \eqref{defnehari} is strictly positive.
\item
The set $ \mathcal{N}_V^G$ is a closed topological manifold of $H^1(\mathbb{R}^N)$ homeomorphic to the unit sphere.   
\end{enumerate}
\end{lemma}

 Lemma \ref{Nehari homeo} is a straightforward adaptation of Lemma 2.8, Proposition 2.9 and Corollary 2.10 in \cite{SzulkinWeth},
as in  our case $E^{+}=H^{1}_{G}$ and $E^{-}=\emptyset$. Then, we just sketch the argument.
\begin{proof} 
For any $u \in H^1_G\setminus \{0\}$
\[ \frac{\langle \mathcal{I}'_V(ru), ru \rangle  }{r^2}= \norm{u}^2_{V} - r^{2p-2} \int_{\mathbb{R}^N} (I_{\alpha} \ast \abs{u}^p) \abs{u}^p 
\]
which is positive for $r>0$ sufficiently small, it goes to $-\infty$ for $r\to +\infty$
and it is strictly decreasing in $r \in (0, \infty)$. 
Then, there exists a unique  $T=T(u)>0$ such that $T(u)u $ is the unique global maximum of $\mathcal{I}_V(tu)$ and  $T(u) u\in \mathcal{N}_V^G $.

The Hardy-Littlewood-Sobolev inequality  immediately implies  $c_V^G >0$. 

In addition,
the map $\widehat m: H^1_G\setminus \{0\} \to \mathcal{N}_V^G$ defined as $\hat m(u) =T(u)u$ is continuous, and its restriction to the unit sphere is a homeomorphism between $S^1$ and $\mathcal{N}_V^G$ because  $\widehat m(u)$ is the unique global maximum of $\mathcal{I}_V$ restricted to the set $\mathbb{R}^+ u$ and $\mathcal{I}_V$ is coercive on $\mathcal{N}_V^G$ as
\[ 
\mathcal{I}_V(u)=\frac12\left(1-\frac{1}{p}\right) \|u\|^{2}_{V}, \quad\forall u\in \mathcal{N}_V^G. \qedhere
\]
\end{proof}
\begin{remark}
Having  defined $\Psi: S^{1}\mapsto \R$ by
\(\Psi(u)=\mathcal{I}_V(\widehat m(u)),  \)
and  following the same arguments as in Proposition 2.9 and Corollary 2.10 in \cite{SzulkinWeth}, it turns out that 
\[
\inf_{S^{+}}\Psi=\inf_{\mathcal{N}_V^G}\mathcal{I}_{V}=c^G_{V}, \quad \text{ where } S^{+}=\{w \in H^1_G : \norm{w}=1\}.
\]
Moreover, $\Psi$ is $C^1$ and \( \Psi'(u)v=T(u) \mathcal{I}_V'(\widehat m(u))v. \)
From this, we deduce that $u$ is  a critical point of $\Psi$ on $S^{+}$ if
and only if $\widehat m(u)$ is a critical point of $\mathcal{I}_{V}$ on $\mathcal{N}_V^G$.
\end{remark}

We are now in the position to detect  the suitable action level where it is possible
to recover a compactness property.
\begin{proposition}\label{estimatesIV}
Let $T_{R}:=T(\chi_{R, z}) $ be defined in Lemma \ref{Nehari homeo} and
assume   \eqref{pless2},  \eqref{Vless2}.
Then, the following  inequality holds 
\[ 
\mathcal{I}_V(T_R \chi_{R, z}) \le \ell(G) c_\infty - \frac12 \sum_{i=1}^{\ell(G)} \sum_{k \ne i} \eps_{ii}^{ik}  + o(\eps_R), \quad \text{ as } R\to +\infty. 
\]
\end{proposition}
\begin{proof}
Let us first notice that,  following Conclusion  (1) of Lemma \ref{Nehari homeo}
it is easy to obtain that $T_{R}:=T(\chi_{R, z}) $ is given by
\begin{equation}\label{eq:T}
T_{R}^{2p-2}
=
\frac{\|\chi_{R, z}\|^{2}_{V}}{\dys\int_{\mathbb{R}^N}   \left(I_{\alpha} \ast \chi^p_{R, 
z} \right)\chi^p_{R, z}}.
\end{equation}
On the other hand, Lemma \ref{estimatesV} and \eqref{defeps}  yield
\[\begin{split}
\|\chi_{R,z}\|^{2}_{V}&\leq 
\sum_{i=1}^{\ell(G)}\|\omega_{i,R}\|^{2}+\sum_{i\neq j}\int_{\R^{N}}\left[\nabla \omega_{i,R}\cdot \nabla \omega_{j,R}+V_{\infty}\omega_{i,R} \omega_{j,R}\right]
+o(\eps_{R})
\\
&=\sum_{i=1}^{\ell(G)}\|\omega_{i,R}\|^{2}+\eps_{R}
+o(\eps_{R}).
\end{split}\]
This, together with  \eqref{eq:T} and Proposition \ref{prop:somma},  implies
\[\begin{split}
\mathcal{I}_{V}(T_R \chi_{R, z})
&=
T_{R}^{2}\left[\frac12\|\chi_{R,z}\|_{V}^{2}-\frac{T_{R}^{2p-2}}{2p}
\int_{\R^{N}}\left(I_{\alpha}\ast\chi^{p}_{R,z}\right)\chi^{p}_{R,z}
\right]
=\left(\frac12-\frac1{2p}\right)\frac{(\|\chi_{R,z}\|_{V}^{2})^{\frac{p}{p-1}}}{\left[\dys\int_{\mathbb{R}^N}   \left(I_{\alpha} \ast \chi^p_{R, 
z} \right)\chi^p_{R, z}\right]^{\frac1{p-1}}}
\\
&\leq
\left(\frac12-\frac1{2p}\right)
\dfrac{\left[\dys \sum_{i=1}^{\ell(G)}\|\omega_{i,R}\|^{2}+\eps_{R}+o(\eps_{R})
\right]^{\frac{p}{p-1}}}{
\left[\dys \sum_{i=1}^{\ell(G)}\|\omega_{i,R}\|^{2}
+p  \varepsilon_R+ p\sum_{i=1}^{\ell(G)} \sum_{k \ne i} \eps_{ii}^{ik}+o(\varepsilon_{R})\right]^{\frac{1}{p-1}}}.
\end{split}
\]
Using the expansion $(a+t)^{\alpha}=a^{\alpha}+\alpha a^{\alpha-1}t+o(t)$  and the notation 
\[
a:= \sum  \|\omega_{i,R}\|^{2}=\ell(G)\|\omega\|^{2},
\] we get
\[\begin{split}
\mathcal{I}_{V}(T_R \chi_{R, z})
&\le \left(\frac12-\frac1{2p}\right)
\left[a+\eps_{R}+o(\eps_{R})\right]^{\frac{p}{p-1}}
\left[a+p  \varepsilon_R+ p \sum_{i=1}^{\ell(G)} \sum_{k \ne i} \eps_{ii}^{ik}+o(\eps_{R})\right]^{-\frac{1}{p-1}}
\\&
=\left(\frac12-\frac1{2p}\right)\left[
a^{\frac{p}{p-1}}+\frac{p}{p-1}a^{\frac1{p-1}}\eps_{R}+o(\eps_{R})\right] 
\\
&
\; \hskip2cm\cdot \left[a^{-\frac{1}{p-1}}-\frac{1}{p-1}a^{-\frac{p}{p-1}}
\Big(p  \varepsilon_R+ p\sum_{i=1}^{\ell(G)} \sum_{k \ne i} \eps_{ii}^{ik}\Big)+o(\eps_{R})\right]
\\
& =
\left(\frac12-\frac1{2p}\right)\left[
a-\frac{p}{p-1}\sum_{i=1}^{\ell(G)} \sum_{k \ne i} \eps_{ii}^{ik} +o(\eps_{R})\right] \\
&=\ell(G) c^G_{\infty}-\frac1{2}\sum_{i=1}^{\ell(G)} \sum_{k \ne i} \eps_{ii}^{ik}+o(\eps_{R}). \qedhere
\end{split} 
 \]
\end{proof}
We are now in the position to prove Theorem \ref{mainthm1}.  
\begin{proof}[Proof of Theorem \ref{mainthm1}]
Let us first prove that
\begin{equation}\label{eq:cGV}
c^{G}_{V}<\ell(G)c^{G}_{\infty}.
\end{equation}
This inequality can be obtained arguing as in estimate \eqref{stima below 
eps}; indeed from \eqref{defeps restr} we infer 
\[
\begin{split}
\eps^{ik}_{ii}&=\iint_{\{|\zeta|\leq \rho R,\,|\theta|\leq \rho R\}}
\frac{\omega^{p}(\zeta)\omega^{p-1}(\theta)}{|\theta-\zeta|^{N-\alpha}}\omega(\theta-R(g_{i}z-g_{k}z))d\zeta d\theta
\\
&
\geq
\iint_{B_{1}(0)\times B_{1}(0)}
\frac{\omega^{p}(\zeta)\omega^{p-1}(\theta)}{|\theta-\zeta|^{N-\alpha}}\omega(\theta-R(g_{i}z-g_{k}z))d\zeta d\theta\\
&\geq
C\inf_{B_{1}(0)\times B_{1}(0)}\frac{\omega^{p}(\zeta)\omega^{p-1}(\theta)}{|\theta-\zeta|^{N-\alpha}}\int_{B_{1}(0)}
\omega(\theta-R(g_{i}z-g_{k}z))  d\theta
\\
&
\geq 
CR^{-\frac{N-\alpha}{2-p}}
\end{split}
\]
where $C>0$ denotes possibly different constants  and the last inequality comes from \eqref{decay}.
This, together with Lemma \ref{est epsilon} and Proposition \ref{estimatesIV} yield \eqref{eq:cGV}.

We can now reach the conclusion arguing as in the  proof of Theorem
1.1 in \cite{SzulkinWeth}: we construct a minimizing Palais-Smale sequence 
for $\mathcal{I}_{V}$, then,  taking into account \eqref{eq:cGV}, we can apply
Proposition 3.1 in  \cite{CingolaniClappSecchi} to deduce that $u_{n}$  is compact.
Therefore, there exists $u \in \mathcal{N}^G_V$ 
such that $\mathcal{I}_V(u)=c_V^G$. As $\abs{u} \in \mathcal{N}^G_V$ too, and
$c_{V}^G=\mathcal{I}_{V}(u)=\mathcal{I}_{V}(|u|)$ we can choose $u$ positive. Hence
by Lemma \ref{Nehari homeo} we have a $G$-invariant positive solution. \end{proof}

\section{Case $p \ge 2$}\label{sec:pgeq2}

This section is devoted to the proof of Theorem \ref{mainthm2} and Theorem 
\ref{mainthm4}. The theoretical strategy of the proof is  analogous to the previous 
section. 
In addition, in this case, the nonlinearities can be treated as in  \cite{clasal} and 
the main  point is to deal with the potential term.
For this range of exponents, the solutions of the limit problem \eqref{Choqlimit} have an exponential decay, so instead of Lemma \ref{lemma cm} we will apply
a result proved in \cite{AmbrosettiColoradoRuiz} (see Lemma \ref{ACR}) 
when $p>2$ or $p=2$ and $\alpha < N-1$.

While, if $p=2$ and $\alpha \in (N-1, N-1/2]$, the solutions of the limit 
problem \eqref{Choqlimit} have an exponential correction, see 
\eqref{decaypertexp},  and we will need to extend Lemma \ref{ACR} in order to treat these different decays, 
(Lemma \ref{lemma estensione1} and \ref{lemma estensione2}). 
These results will allow us to  prove that also in this situation the integral involving the potential decays faster. 
Then, the proofs of  Theorem \ref{mainthm2} and Theorem 
\ref{mainthm4} will be given in Subsection  \ref{proofs2}. 

\subsection{Asymptotic Analysis} \label{Extension}
In this Subsection we first  prove two extensions of Lemma \ref{ACR} to functions with an exponential correction in the decays. The proof, which is partly inspired by \cite{AmbrosettiColoradoRuiz}, requires a very careful analysis, and we will need to split it into two different Lemma, proved arguing in different ways depending on the coefficients. 
Thanks to these Lemma we will be able to perform the asymptotic study 
as in subsection \ref{sec:estimates1}.
\begin{lemma}\label{lemma estensione1}
Let $u, v$ be two continuous, positive radial functions such that 
\begin{equation}\label{eq:ipodec}
 u \sim \abs{x}^{a} e^{-b \abs{x}+c \abs{x}^{\gamma} } \qquad
 v \sim \abs{x}^{a'} e^{-b' \abs{x}+c' \abs{x}^{\gamma'} }, \quad\text{as $\abs{x} \to \infty,$}
\end{equation}
where $b, b',c, c'>0$, $a, a' \in \R$, $\gamma \in (0, 1)$, and  $\gamma' \in [0,1)$. 
Then the following estimate holds
\[ 
\int_{\R^N} u_\xi v \sim 
\abs{\xi}^{a} e^{-b\abs{\xi}+ c \abs{\xi}^\gamma} \qquad \text{ if } b< b', \text{ or if } b=b' \text{ and } \gamma > \gamma', 
\] 
where  $u_{\xi}(x)=u(x-\xi)$. 
\end{lemma}

\begin{proof}
The proof is quite lengthy, so that it will be divided into steps.
\\
{\bf Step 1.}
We preliminarily give a bound from below.
By the positivity and the continuity of  $v$ one has that $v\geq C>0$ on $B_{1}(0)$, the ball
centred at zero with radius one, this together with  the fact that the  function 
$f(t)=t^a e^{-b t+c t^\gamma}$ is decreasing if $t$ is sufficiently large,
yields
\begin{equation}\label{stima below} 
\begin{split}
\int_{\R^N} u_\xi v 
&\ge \int_{B_{1}(0)} u_\xi v \ge C \int_{B_{1}(0)} \abs{x-\xi}^{a} e^{-b \abs{x-\xi}+c \abs{x-\xi}^{\gamma} }\, dx  
\\ 
&
\ge C (\abs{\xi}+1)^a e^{-b \abs{\xi}-b + c(\abs{\xi}+1)^\gamma} \sim\underline{ C} \abs{\xi}^a e^{-b \abs{\xi} + c\abs{\xi}^\gamma}.
\end{split}
\end{equation}
{\bf Step 2.}
In this step we will show that
\begin{equation}\label{eq:acr1}
\int_{\R^N} u_\xi v \leq \int_{r_0}^{\xi_0-r_0}dr\int_{\R^{N-1}}u_\xi vdy+
 C\left[\xi_0^{a}e^{-b\xi_0+ c \xi_0^{\gamma}}+\xi_0^{a'}e^{-b'\xi_0+ c' \xi_0^{\gamma'}}\right],
\end{equation}
where  
\begin{equation}\label{eq:xizero}
x= (r, y) \in \mathbb{R} \times \mathbb{R}^{N-1},\qquad
\xi= (\xi_0, 0, \dots, 0), \qquad \text{and }\; r_{0}\in (1,\xi_{0}/2).
\end{equation}
Let us observe that $r_{0}$ will be fixed sufficiently large and 
 the notation on $\xi$ can be taken  up to rotations.
In order to prove \eqref{eq:acr1},  we split the integral as follows
\begin{equation}\label{into} 
\int_{\R^N} u_\xi v = \int_{-\infty}^{r_0}dr\int_{\R^{N-1}}u_\xi v dy
+\int_{r_0}^{\xi_0-r_0}dr\int_{\R^{N-1}}u_\xi vdy
+ \int_{\xi_{0}-r_0}^{+\infty}dr\int_{\R^{N-1}} u_\xi vdy  .
\end{equation}
As $u,\,v$ are radial functions, by performing the change of variables
$r'=\xi_{0}-r$ one gets
\[
\int_{-\infty}^{r_0}dr\int_{\R^{N-1}}u_{\xi}v dy+\int_{\xi_{0}-r_0}^{+\infty}dr\int_{\R^{N-1}} u_{\xi}vdy
=\int_{-\infty}^{r_0}dr\int_{\R^{N-1}}(u_{\xi}v+uv_{\xi}) dy.
\]
Now, note that for every $r<r_0$ it results
\[
\abs{x-\xi}=\sqrt{(\xi_{0}-r)^{2}+|y|^{2}}\geq |\xi_{0}-r|> \xi_0-r_0,
\]
then thanks to the monotonicity properties of the function $f(t)=t^a e^{-b t+c t^\gamma}$ already observed, one deduces that, for $\xi_0$ sufficiently  large there exists a positive constant $C$ such that
\[ 
|x-\xi|^{a}e^{-b\abs{x-\xi}+ c \abs{x-\xi}^{\gamma}}\leq C \xi_{0}^a e^{-b \xi_0+ c\xi_0^{\gamma}},
\quad
|x-\xi|^{a'}e^{-b'\abs{x-\xi}+ c' \abs{x-\xi}^{\gamma'}}\leq C \xi_{0}^{a'} e^{-b' \xi_0+ c'\xi_0^{\gamma'}}.
\]
These facts and \eqref{eq:ipodec}  yield (with $C$ possibly different constants)
\begin{equation}\label{into1stima}
\begin{split}
\int_{-\infty}^{r_0}\hskip-5pt dr\hskip-5pt\int_{\R^{N-1}}\hskip-5pt 
\left(u_{\xi}v+uv_{\xi}\right) dy
\leq &
C \xi_0^{a}e^{-b\xi_0+ c \xi_0^{\gamma}}
\int_{-\infty}^{r_0}\hskip-5pt dr\hskip-5pt\int_{\R^{N-1}}  
vdy
\\&
+ C \xi_0^{a'}e^{-b'\xi_0+ c' \xi_0^{\gamma'}}
\int_{-\infty}^{r_0}\hskip-5pt dr\hskip-5pt\int_{\R^{N-1}}  
udy
\\
\leq &  C\left[\xi_0^{a}e^{-b\xi_0+ c \xi_0^{\gamma}}+\xi_0^{a'}e^{-b'\xi_0+ c' \xi_0^{\gamma'}}\right]
\end{split}
\end{equation}
where the last inequality is deduced observing that 
$u,\,v\in L^{1}(\R^{N})$, so that \eqref{eq:acr1} holds.
\\
{\bf Step 3.}
In this step we are going to show that
\begin{equation}\label{into2 stima} 
\int_{r_0}^{\xi_0-r_0}dr \int_{\R^{N-1}} u_\xi v \, dy \le C\,\xi_0^a e^{-b \xi_0 + c \xi_0^\gamma}, \qquad \text{if $b<b'$.}
\end{equation}
First of all, we can find $\tilde b,\,C$   positive constants  such that
\begin{equation}\label{eq:moduli}
e^{-b' \abs{x} - b \abs{\xi-x} + c' \abs{x}^{\gamma'} + c \abs{\xi-x}^\gamma} \leq 
e^{-b\xi_0-(b'-b) \abs{x} + c' \abs{x}^{\gamma'} + c \abs{\xi-x}^\gamma} 
\leq
Ce^{-b\xi_0-\tilde b \abs{x}+ c \xi_0^\gamma} .
\end{equation}
In addition,  taking into account \eqref{eq:xizero}  it results
\[
\abs{x}^{a'} \abs{\xi-x}^a\leq r^{a'} \abs{\xi_0-r}^a,\quad \text{when $a',\,a<0$ },
\]
while, if both $a$ and $a'$ are positive, by direct computations, we can find a positive constant $C$
such that 
\[
\abs{x}^{a'} \abs{\xi-x}^a\leq Cr^{a'} \abs{\xi_0-r}^a (1+\abs{y})^{a'} (1+\abs{y})^{a}.
\]
Then, noting that $|x|\geq \frac{r+|y|}2$, and using \eqref{eq:moduli}, we obtain
(denoting with $C$ possibly different constants)
\[
\begin{split} 
\int_{r_0}^{\xi_0-r_0}dr \int_{\R^{N-1}} u_\xi v \, dy 
&
\le C e^{-b\xi_0+ c \xi_0^\gamma}
\int_{r_0}^{\xi_0-r_0}dr \int_{\R^{N-1}} r^{a'} \abs{\xi_0-r}^a e^{-\tilde b \frac r2 
} e^{-\tilde b\frac{\abs{y}}{2}}\tilde{h}(y) \, dy 
\\
&\le C\,e^{-b \xi_0 + c \xi_0^\gamma} \int_{r_0}^{\xi_0-r_0} r^{a'} \abs{\xi_0-r}^a e^{-\tilde b \frac r2} \,dr, 
\end{split}
\]
where we have used that $e^{-\tilde b\frac{\abs{y}}{2}}\tilde{h}(y)\in L^{1}(\R^{N-1})$. 
When $a$ and $a'$ have opposite sign an analogous argument leads to the same conclusion.
This last integral can be now estimated exactly as in \cite{AmbrosettiColoradoRuiz} (Lemma 3.7 pp. 108-109), and we get \eqref{into2 stima}.
\\
{\bf Step 4. }
In  this step we will consider the case $b=b'$ and we will  show that
\begin{equation}\label{into2 stimab=} 
\begin{split}
\int_{r_0}^{\xi_0-r_0}dr \int_{\R^{N-1}} u_\xi v \, dy 
\le &
 \int_{r_0}^{\frac{\xi_0}{2}}dr\int_{\{\abs{y}<r\}}(u_\xi v + u v_\xi)  dy
\\&
+
C e^{-b\xi_{0}}\left[\xi_0^{a'} e^{c' \xi_0^{\gamma'}}+\xi_0^a e^{c \xi_0^{\gamma}} \right]. 
\end{split}
\end{equation}
Performing  the change of variables $r'=r-\xi_{0}$ and taking into account the symmetry  properties of $u$ and $v$, one gets
\[
\int_{\xi_{0}/2}^{\xi_{0}-r_{0}}dr\int_{\R^{N-1}}u_{\xi}vdy=
\int_{r_{0}}^{\xi_{0}/2}dr\int_{\R^{N-1}}uv_{\xi}dy
\]
so that
\begin{equation}\label{into2}
\begin{split}
\int_{r_0}^{\xi_0-r_0}dr\int_{\R^{N-1}} u_\xi v dy 
= &   
\int_{r_0}^{\frac{\xi_0}{2}}dr\int_{\abs{y}>r}(u_\xi v + u v_\xi)  dy 
\\&
+ \int_{r_0}^{\frac{\xi_0}{2}}dr\int_{\abs{y}<r}(u_\xi v + u v_\xi)  dy.
\end{split}
\end{equation}
Then, in order to show \eqref{into2 stimab=} 
we have to study  the first integral on the right hand side.

Notice that $\abs{x} \le r + \abs{y}$, and $\abs{\xi-x} \le \xi_0 + \abs{y}$, 
so that as $\gamma,\,\gamma'\in [0,1)$, $\abs{x}^{\gamma'} \le r ^{\gamma'}+ \abs{y}^{\gamma'} \le 2 \abs{y}^{\gamma'}$, and $\abs{\xi-x}^{\gamma} \le \xi_0^{\gamma} + \abs{y}^{\gamma}$.
Moreover, as $|y|>r$ in the integral under study, it holds that
 $\abs{x}\geq \sqrt2 r$ and $(1-1/\sqrt{2})|x|\geq (1-1/\sqrt{2})|y|$,
summing up, one gets
$\abs{x} \ge r +(1-1/\sqrt{2})  \abs{y}$.  
Thus
\begin{equation}\label{eq:lambda} 
\abs{x} +  \abs{\xi-x} > \xi_0 +  \lambda \abs{y},\qquad  \lambda =1-1/\sqrt{2},
 \end{equation}
yielding
\begin{equation}\label{eq:exp}
e^{-b(|x-\xi|+|x|)+c|x-\xi|^{\gamma}+c'|x|^{\gamma'}}\leq 
e^{-b \xi_0+c \xi_0^{\gamma}} e^{-\lambda b  \abs{y}+c\abs{y}^{\gamma}+ 2c' \abs{y}^{\gamma'}}.
\end{equation}
Moreover, $\abs{x} \ge r$, and $\abs{\xi-x} \ge \frac{\xi_0}{2}$,  so that
\begin{align*}
|x-\xi|^{a}|x|^{a'}&\leq C|r|^{a'}\xi_{0}^{a} \le C \xi_0^a,\hskip3.25cm\text{if  $a,\,a'<0$,}
\\
|x-\xi|^{a}|x|^{a'}&\leq C|y|^{a'}(\xi_{0}^{a}+|y|^{a})\leq C \xi_{0}^{a}|y|^{a+a'},\qquad \text{if  $a,\,a'>0$},
\end{align*}
where  the last inequality follows from  the fact that $|y|\geq r$.
In the case in which $a$ and $a'$ have opposite sign we will obtain 
a combination of the previous estimates.
This, together with  \eqref{eq:ipodec} and \eqref{eq:exp}, implies
\[
\begin{split}
\int_{r_0}^{\frac{\xi_0}{2}} \int_{\abs{y}>r}  u_{\xi} v &
\leq 
C \xi_0^a e^{-b \xi_0+c \xi_0^{\gamma}} \int_{r_0}^{\frac{\xi_0}{2}} \int_{\abs{y}>r} 
|y|^{|a'|+|a|} e^{-\lambda b  \abs{y}+c\abs{y}^{\gamma}+ 2c' \abs{y}^{\gamma'}}
\\
&\le 
C \xi_0^a e^{-b \xi_0+c \xi_0^{\gamma}} 
\int_{r_0}^{\infty}\int_r^{\infty} \rho^{|a|+|a'|+N-2} e^{-\lambda b  \rho+ c \rho^{\gamma} +2c' \rho^{\gamma'}} \, dr \, d \rho
\\
&=
C \xi_0^a e^{-b \xi_0+c \xi_0^{\gamma}} 
\int_{r_{0}}^{+\infty}d\rho \int_{r_{0}}^{\rho}\rho^{|a|+|a'|}\rho^{N-2} e^{-\lambda b  \rho+ c \rho^{\gamma} +2c' \rho^{\gamma'}} \,dr
\\
&\leq C \xi_0^a e^{-b \xi_0+c \xi_0^{\gamma}} 
\int_{r_0}^{\infty}\rho^{N-1+|a|+|a'|}  e^{-\lambda b \rho + c \rho^{\gamma} +2c' \rho^{\gamma'}} \, d \rho 
\\
&\leq C \xi_0^a e^{-b \xi_0+c \xi_0^{\gamma}} ,
\end{split} \]
where we recall that $\lambda >0$ is introduced in \eqref{eq:lambda}.
Arguing analogously, one obtains \eqref{into2 stimab=}.
\\
{\bf Step 5.}
In this step we will prove 
\begin{equation}\label{estimate2}  
\begin{split}
\int_{r_0}^{\frac{\xi_0}{2}} dr\int_{\abs{y}<r}  (u_{\xi} v+u v_\xi ) dy 
\sim &
\xi_0^{a}e^{-b \xi_0}    \int_{r_0}^{\frac{\xi_0}{2}} r^{\frac{N-1}{2} + a'} e^{c'r^{\gamma'} + c(\xi_0 - r)^{\gamma}}dr
\\
&+\xi_0^{a'}e^{-b \xi_0}  \int_{r_0}^{\frac{\xi_0}{2}} r^{\frac{N-1}{2} + a} e^{cr^{\gamma} + c'(\xi_0 - r)^{\gamma'}}dr. 
\end{split}\end{equation}
Recalling \eqref{eq:xizero} and taking into consideration that $|y|\leq r$, one has
\begin{equation}\label{dismod}
 r < \abs{x} = \sqrt{r^2 +\abs{y}^2} < \sqrt 2 r,\qquad 
 \frac{ \xi_0}{2} < \abs{\xi - x} <\abs{\xi_0 -r} + r = \xi_0. 
 \end{equation}
Moreover, let us take $h=|y|^{2}$ and consider, for every $s, t \in (0, 2)$, and $d, d' > 0$, the function
\[
f(h):=d'\abs{x}^{s}+ d\abs{\xi-x}^{t}= d'(r^{2}+h)^{\frac{s}2} + d\left((\xi_{0}-r)^{2}+h\right)^{\frac{t}2}.
\]
Then, \eqref{dismod} yields for $h\in [0,r^{2}]$ and $\xi_{0}>2r$
\[ 
 d' \frac{s}2 2^{\frac{s}2 -1} r^{s-2} \le f'(h) \le   \left( d \frac t 2 r^{t-2} + d' \frac{s}2 r^{s-2} \right), \]
so that
\[ d' \frac{s}2 2^{\frac{s}2 -1} r^{s-2} h  + f(0) \le f(h) \le f(0) + \left( d \frac t 2 r^{t-2} + d' \frac{s}2 r^{s-2} \right) h, \qquad \forall\, h\in [0,r^{2}].
\]
As a consequence, the following inequality holds for every $s, t \in (0, 2)$ 
\[
\begin{split}
d' r^s + d (\xi_0-r)^{t} +  d' \frac{s}2 2^{\frac{s}2 -1}  r^{s-2} |y|^{2}
& < d'\abs{x}^{s}+ d\abs{\xi-x}^{t}
 \\
&< d' r^s + d (\xi_0-r)^{t} + \left( d \frac t 2 r^{t-2} + d' \frac{s}2 r^{s-2} \right) |y|^{2}.
\end{split}
\]
Using these information (both for $s=\gamma'$, $t=\gamma$, $d=c$, $d'=c'$, and for $s=t=1$, $d=d'=b$), together with \eqref{eq:ipodec} and \eqref{dismod}, we obtain
\[ \begin{split}
\int_{\abs{y}<r}  u_{\xi} vdy \sim 
 \xi_0^{a} e^{-b \xi_0}  r^{a'} e^{c' r^{\gamma'} + c (\xi_0 -r)^{\gamma} } \int_{\abs{y}
 <r}  e^{-\hat b_1 \frac{\abs{y}^2}{r} + \hat b_2 \frac{\abs{y}^2}{r^{2-\gamma}} + \hat b_3 \frac{\abs{y}^2}{r^{2-\gamma'}} }dy, 
\end{split}\]
where $\hat b_1, \hat b_2, \hat b_3$ are positive constants which depend on the parameters and on whether we are considering estimates from above or below. 
Notice that, for every $r>r_0$ and for  $r_{0}$  fixed sufficiently large, 
(depending on the parameters but not on $\xi_{0}$), 
\[ 
\frac {\hat b_1}2 \frac{\abs{y}^2}{r} 
\le \frac{\abs{y}^2}{r}\left(\hat b_1 -\frac{\hat b_2}{r_0^{1-\gamma }}-\frac{\hat b_3}{r_0^{1-\gamma' }}\right) 
\le \hat b_1 \frac{\abs{y}^2}{r}. \]
 Therefore, choosing   $\hat b=\hat b_1$ in the  estimate from below  and $\hat b=\hat b_1 /2$ in the one from above, it follows
\[
\begin{split}
\int_{r_0}^{\frac{\xi_0}{2}} dr\int_{\abs{y}<r}  u_{\xi} v dy 
&\sim 
\xi_0^a e^{-b \xi_0}  \int_{r_0}^{\frac{\xi_0}{2}} r^{a'} e^{c'r^{\gamma'} + c(\xi_0 - r)^{\gamma}}  dr \int_{\abs{y}<r}  e^{-\hat b \frac{|y|^2}{r} }dy
\\
&=\xi_0^a e^{-b \xi_0}  \int_{r_0}^{\frac{\xi_0}{2}} r^{a'} r^{\frac{N-1}{2}} e^{c'r^{\gamma'} + c(\xi_0 - r)^{\gamma}}  dr \int_{\abs{y'}<\sqrt{r}}  e^{-\hat b| y'|^{2} }dy'.
\end{split}\]
The last integral is bounded from above by the integral in the whole
$\R^{N-1}$ and, since $r>1$, it is bounded from below by the integral on 
$B_{1}(0)$, both finite;
so that it results
\begin{equation}\label{eq:stima per below}
\int_{r_0}^{\frac{\xi_0}{2}}dr \int_{\abs{y}<r}  u_{\xi} v \sim 
\xi_0^{a} e^{-b \xi_0}  \int_{r_0}^{\frac{\xi_0}{2}} r^{\frac{N-1}{2} + a'} e^{c'r^{\gamma'} + c(\xi_0 - r)^{\gamma}}dr. 
\end{equation}
By similar computations, exchanging the role of the coefficients, one has 
\eqref{estimate2}.
\\
{\bf Step 6.}
In this step we will conclude the proof.

Let us start dealing with the first integral on the right hand side of \eqref{estimate2}, and
notice that on the interval $[r_0, \xi_0/2]$ one has $r^{\gamma'} < (\xi_0/2)^{\gamma'}$ and $(\xi_0 - r)^\gamma \le (\xi_0-r_0)^\gamma$. Hence for any $\gamma > \gamma'' >\gamma' \ge 0$ one has 
\[
\begin{split} 
\int_{r_0}^{\frac{\xi_0}{2}} r^{\frac{N-1}{2} + a'} e^{c'r^{\gamma'} + c(\xi_0 - r)^{\gamma}}dr 
&
\le e^{c'(\xi_0/2)^{\gamma'} + c(\xi_0 - r_0)^\gamma + (\xi_0/2)^{\gamma''}} \int_{r_0}^{\frac{\xi_0}{2}} r^{\frac{N-1}{2} + a'} e^{- r^{\gamma''}} \, dr 
\\&
\le C e^{c'(\xi_0/2)^{\gamma'} +  (\xi_0/2)^{\gamma''} + c(\xi_0 - r_0)^\gamma}. \end{split}
\]
As $\gamma > \gamma'' >\gamma' \ge 0$, one gets
 \( e^{c'(\xi_0/2)^{\gamma'} +  (\xi_0/2)^{\gamma''} + c(\xi_0 - r_0)^\gamma} \sim e^{c \xi_0^\gamma} \)
if $\xi_0$ is sufficiently large, so that
\begin{equation}\label{estimate4} \int_{r_0}^{\frac{\xi_0}{2}} r^{\frac{N-1}{2} + a'} e^{c'r^{\gamma'} + c(\xi_0 - r)^{\gamma}}dr \le C \, e^{c \xi_0^\gamma}. \end{equation}
When arguing on the second integral in \eqref{estimate2}, we 
take into account that the role of $\gamma $ and $\gamma'$ are
exchanged, and as $\gamma>\gamma'$  we obtain
\begin{equation}\label{estimate4 bis} \int_{r_0}^{\frac{\xi_0}{2}} r^{\frac{N-1}{2} + a} e^{cr^{\gamma} + c'(\xi_0 - r)^{\gamma'}}dr \le C \, e^{c(\xi_0/2)^\gamma},
\qquad \text{if $\xi_0$ is big enough.}
\end{equation}

Finally,  exploiting \eqref{stima below}, \eqref{eq:acr1}, \eqref{into2 stima} 
one gets the conclusion if $b<b'$.
When $b=b'$ and $\gamma>\gamma'$, one takes into
account \eqref{stima below}, \eqref{eq:acr1}, \eqref{into2 stimab=},
\eqref{estimate2},  \eqref{estimate4} and \eqref{estimate4 bis} to conclude the proof.
\end{proof}
In the next lemma we analyze the case $b=b'$ and $\gamma= \gamma'$, concluding the extension of Lemma \ref{ACR} useful  in our context.
\begin{lemma}\label{lemma estensione2}
Let $u, v$ be two continuous, positive radial functions such that \eqref{eq:ipodec} is satisfied with $b'=b>0$, $\gamma'=\gamma\in (0,1)$,  $ c, c'>0$ and $a, a' \in \R$. 
Then the following estimate holds
\[ 
\int_{\R^N} u_\xi v \sim 
\abs{\xi}^{\frac{N+1}{2} + a+ a'-\frac\gamma2} e^{-b\abs{\xi} + \tilde c \abs{\xi}^\gamma},\qquad \text{with }\, \tilde c= \left( (c')^{\frac{1}{1-\gamma}} + c^{\frac{1}{1-\gamma}} \right)^{1-\gamma},
\] 
where  $u_{\xi}(x)=u(x-\xi)$. 
\end{lemma}
\begin{remark}\label{rem:c'}
Let us observe that in  Lemma \ref{lemma estensione1}, 
\ref{lemma estensione2}, we assume $c,\,c'> 0$ or $\gamma,\,\gamma'> 0$ as the 
cases $c,\,c'\leq 0$ or $\gamma,\,\gamma'\leq 0$ are already contained in  
Lemma \ref{ACR}. 

Moreover,  the case $c'\geq 0,\, \gamma'>0$ is 
equivalent to our assumption $c'>0,\, \gamma'\geq 0$, and 
 if $b'>b$ the result can be proved as well, by exchanging the role of $b,\,b'$ and $\gamma,\,\gamma'$.
\end{remark}

\begin{proof} 
Let us start proving estimates from above and assuming, without loss of generality, that $c \ge c'>0$. 
As in the proof of Lemma \ref{lemma estensione1}, we use the notation
in \eqref{eq:xizero} with $r_{0}$ such that $1<r_{0}<(1+(c/c')^{1/(1-\gamma)})^{-1}\xi_{0}/2<\xi_{0}/2$.
Steps 1, 2, 4, 5 in the proof of Lemma \ref{lemma estensione1} are still valid,
hence we can take into account 
\eqref{eq:acr1}, \eqref{into2 stimab=}, \eqref{estimate2}, and obtain
\begin{equation}\label{eq:est2uno}
\begin{split}
\int_{\R^N} u_\xi v 
\leq &
C
\xi_0^{a}e^{-b \xi_0}    \int_{r_0}^{\frac{\xi_0}{2}} r^{\frac{N-1}{2} + a'} e^{c'r^{\gamma} + c(\xi_0 - r)^{\gamma}}dr
\\
&
+C\xi_0^{a'}e^{-b \xi_0}  \int_{r_0}^{\frac{\xi_0}{2}} r^{\frac{N-1}{2} + a} e^{cr^{\gamma} + c'(\xi_0 - r)^{\gamma}}dr
+C\xi_0^a e^{-b\xi_{0}+c \xi_0^{\gamma}} .
\end{split}
\end{equation}
However, the estimate of the integrals in the right hand side of \eqref{eq:est2uno}, which in case $\gamma>\gamma'$ corresponds to Step 6 in the proof of Lemma \ref{lemma estensione1}, requires a more accurate analysis when considering the case $\gamma=\gamma'$. 

 Let us define
\begin{equation}\label{eq:defg} 
g(r):=c' r^{\gamma} + c (\xi_0 - r)^{\gamma} 
\end{equation}
and observe that it has a maximum in the interval $[r_0, \xi_0/2]$ in the point 
$\hat r$,
where 
\begin{equation}\label{eq:rhat}
\hat r = \frac{\xi_0}{1+\left(c/c'\right)^{\frac{1}{1-\gamma}}}
\quad\text{and}\quad
g(\hat r)=\xi_0^\gamma \left( (c')^{\frac{1}{1-\gamma}} + c^{\frac{1}{1-\gamma}} \right)^{1-\gamma}. 
\end{equation}
As $c\geq c'$,   $\hat r\leq  \xi_0/2$; then we can write
\begin{equation}\label{ultima divisione} 
\begin{split}
\int_{r_0}^{\xi_0/2} 
r^{\frac{N-1}{2} + a'}e^{c' r^{\gamma} + c (\xi_0 - r)^{\gamma}} 
=&
 \int_{r_0}^{\hat r/2}  r^{\frac{N-1}{2} + a'}e^{c' r^{\gamma} + c (\xi_0 - r)^{\gamma}} \, dr
\\& 
+ \int_{\hat r/2}^{\xi_0/2} 
 r^{\frac{N-1}{2} + a'}e^{c' r^{\gamma} + c (\xi_0 - r)^{\gamma}} \, dr
\end{split} 
\end{equation}
and, noting that $g(r)\leq g(\hat{r}/2)$ in $(r_{0},\hat{r}/2)$, one obtains
\begin{align}\label{stima 6}  
\int_{r_0}^{\hat r/2}  r^{\frac{N-1}{2} + a'}e^{c' r^{\gamma} + c (\xi_0 - r)^{\gamma}} \, dr 
& 
\le e^{g(\hat r/2)} \int_{r_0}^{\hat r/2}  r^{\frac{N-1}{2} + a'}\, dr 
\le C \xi_0^{\frac{N+1}{2} + a'} e^{g(\hat r/2)}.
\end{align}
In order to estimate the second integral in \eqref{ultima divisione} we need to study the behavior of $g$ near $\hat r$. By Taylor expansion at the 
maximum point $\hat r$,  one has
\begin{equation}\label{taylor} 
g(r)=g(\hat r)+\frac12 g''(r_1)(\hat r- r)^2 \end{equation}
where $r_1$ belongs to the interval of extrema $ r$ and $\hat r$ so that
$r_1 \in (\hat{r}/2, \xi_0/2)$.
In addition, 
\begin{equation}\label{eq:gh}
g''(r)=\gamma(\gamma-1)h(r)\qquad \text{where}\qquad h(t)=\frac{c'}{t^{2-\gamma}} + \frac{c}{(\xi_0-t)^{2-\gamma}} 
\end{equation}
and $h(t)$ has a global minimum point at $\hat t$ such that
\[ \frac{\xi_0}{2} \ge \hat t=\frac{\xi_0}{1+\left(c/c'\right)^{\frac{1}{3-\gamma}}} \ge \hat r, \qquad \text{as $c \ge c'$.  }
 \]
Hence, taking into account that $\gamma<1$,  \eqref{eq:gh} together
with \eqref{taylor} yields
\begin{equation}\label{stima g} 
g(r) \le g(\hat r) - \frac12 c_1 \gamma(1-\gamma)  \xi_0^{\gamma-2} (\hat r- r)^2. \end{equation}
Exploiting \eqref{stima g} into the second integral on the right hand side of \eqref{ultima divisione}, one has
\begin{equation}\label{stima 5}   
\begin{split}
\int_{\hat r/2}^{\xi_0/2}  r^{\frac{N-1}{2} + a'}e^{c' r^{\gamma} + c (\xi_0 - r)^{\gamma}} \, dr  
\le &
  e^{g(\hat r)} \int_{\hat r/2}^{\xi_0/2} r^{\frac{N-1}{2} + a'}e^{-\hat c \xi_0^{\gamma-2} (\hat r- r)^2} \, dr
\\
=&e^{g(\hat r)} \int_{\hat r/2}^{\hat r} r^{\frac{N-1}{2} + a'}e^{-\hat c \xi_0^{\gamma-2} (\hat r- r)^2} \, dr
\\
&+e^{g(\hat r)} \int_{\hat r }^{\xi_0/2} r^{\frac{N-1}{2} + a'}e^{-\hat c \xi_0^{\gamma-2} (\hat r- r)^2} \, dr
\end{split}\end{equation}
where $\hat c= \frac12 c_1 \gamma(1-\gamma) $. 
Let us study  the first integral on the right hand side and note that
 \eqref{eq:rhat} implies that there exist two positive constants
 $C_{1}<C_{2}$ such that
  \begin{equation}\label{eq:rstima}
C_{1} \xi_0^{\frac{N-1}{2} + a'} \le  r^{\frac{N-1}{2} + a'} \le C_{2} \xi_0^{\frac{N-1}{2} + a'},
\qquad \forall \, r\in [\hat{r}/2,\xi_{0}/2].
\end{equation}
Performing the change of
variables $w=(\hat c\xi_0^{\gamma-2})^{1/2} (\hat r-r)$, one obtains
\begin{align*} 
\int_{\hat r/2}^{\hat r} r^{\frac{N-1}{2} + a'}e^{-\hat c \xi_0^{\gamma-2} (\hat r- r)^2}\, dr 
&\le 
C_{2} \, 
\xi_0^{\frac{N+1}{2} + a'- \frac \gamma2}
\int_{0}^{\frac{\sqrt{\hat c}}{2} \xi_0^{\gamma/2-1}\hat r} 
e^{- w^2 } \, dw 
\\ 
&\le C\, \xi_0^{\frac{N+1}{2} + a'- \frac \gamma2} \int_0^{+\infty} e^{- w^2 } \, dw \le C \, \xi_0^{\frac{N+1}{2} + a'- \frac \gamma2}. 
\end{align*}
Using again \eqref{eq:rstima} in the second integral and performing  the change of variables $w=(\hat c\xi_0^{\gamma-2})^{1/2} (r-\hat r )$, one deduces that
\[
\int_{\hat r}^{\xi_0/2} r^{\frac{N-1}{2} + a'}e^{-\hat c \xi_0^{\gamma-2} (r- \hat r)^2}\, dr \le C \, \xi_0^{\frac{N+1}{2} + a'- \frac \gamma2},
\]
so that, \eqref{stima 5} becomes
\[
\int_{\hat r/2}^{\xi_0/2}  r^{\frac{N-1}{2} + a'}e^{c' r^{\gamma} + c (\xi_0 - r)^{\gamma}} \, dr
\leq 
C \xi_0^{\frac{N+1}{2} + a'- \frac \gamma2}e^{g(\hat r)}.
\]
Exploiting this last information and \eqref{stima 6} together with the fact that $g(\hat{r}/2)<g(\hat{r})$ into  \eqref{ultima divisione}  one obtains
\[
\int_{r_0}^{\xi_0/2} 
r^{\frac{N-1}{2} + a'}e^{c' r^{\gamma} + c (\xi_0 - r)^{\gamma}} 
\leq
C \xi_0^{\frac{N+1}{2} + a'- \frac \gamma2}e^{g(\hat r)}.
\]
As a consequence, from  \eqref{eq:est2uno},
and taking into account that $g(\hat{r})> c\xi_{0}^{\gamma}$ (thanks to 
\eqref{eq:rhat}),  we deduce that
\begin{equation}\label{eq:est2due}
\begin{split}
\int_{\R^N} u_\xi v 
\leq &
C\xi_0^{\frac{N+1}2+a+a'-\frac{\gamma}2}e^{-b \xi_0+g(\hat{r})}  
+C\xi_0^{a'}e^{-b \xi_0}  \int_{r_0}^{\frac{\xi_0}{2}} r^{\frac{N-1}{2} + a} e^{cr^{\gamma} + c'(\xi_0 - r)^{\gamma}}dr,
\end{split}
\end{equation}
for $c\geq c'$.
Notice that, if $c = c'$, in order to study the integral on the right hand side
we can repeat  the argument above exchanging the roles of $a$ and $a'$,
and  we obtain
\begin{equation}\label{eq:stimaabove}
\int_{\R^N} u_\xi v \le 
\xi_0^{\frac{N+1}{2} + a+ a'-\frac\gamma2} e^{-b\xi_0 + g(\hat r)},
\end{equation}
and the proof of estimates from above in this case is complete, recalling \eqref{eq:rhat}. 

On the other hand, if $c>c'$, then we need to estimate differently the integral appearing on the right hand side of \eqref{eq:est2due}. We consider the function
\[ 
\tilde g(r):=g(\xi_0-r)=c r^\gamma + c'(\xi_0-r)^\gamma
\]
and notice that it is increasing  in the interval $[r_0, \xi_0/2]$, so that 
\[ 
\int_{r_0}^{\frac{\xi_0}{2}} r^{\frac{N-1}{2} + a} e^{cr^{\gamma} + c'(\xi_0 - r)^{\gamma}}dr \le C \, e^{\tilde g(\xi_0/2) } 
\int_{r_0}^{\frac{\xi_0}{2}} r^{\frac{N-1}{2} + a}dr. \] 
In addition, $\hat r < \xi_0/2$ as $c > c'$ (see \eqref{eq:rhat}), so that
$ \tilde g(\xi_0/2)=g(\xi_0/2) < g(\hat r)$, as $\hat{r}$ is the maximum point of 
$g$ in $[r_{0},\xi_{0}/2]$.
Hence, integrating one has
\[ \int_{r_0}^{\frac{\xi_0}{2}} r^{\frac{N-1}{2} + a} e^{cr^{\gamma} + c'(\xi_0 - r)^{\gamma}}dr \le C \, \xi_0^{\frac{N+1}{2} +a- \frac \gamma2} e^{g(\hat r) }. \] 
Therefore, taking  into account \eqref{eq:est2due}, one has that
 \eqref{eq:stimaabove}  holds in this case too. 

To conclude the proof, we need to prove the estimate from below. 
We notice that on the interval $[\hat r/2, \hat r]$ the function $h$ given in 
\eqref{eq:gh}, attains its maximum on the boundary. In any case, one has
\[ h(r_1) \le c_2 \xi_0^{\gamma-2}, \quad \text{ where } r_1 \in \left(\hat r/2, \hat r \right), \]
with a suitable constant $c_2>0$.
Exploiting a second order Taylor expansion of the function $g$ as in  \eqref{taylor}
one deduces that
\begin{equation}\label{stima 7}
\begin{split} 
\int_{\hat r/2}^{\hat r} r^{\frac{N-1}{2} + a'}e^{c' r^{\gamma} + 
c (\xi_0 - r)^{\gamma}} dr &\ge 
\int_{\hat r/2}^{\hat r}  r^{\frac{N-1}{2} + a'}e^{g(\hat r) - \frac12 c_2 \gamma(1-\gamma)  \xi_0^{\gamma-2} (\hat r - r)^2 } dr
\\ &= 
e^{g(\hat r)} \int_{\hat r/2}^{\hat r}  r^{\frac{N-1}{2} + a'}e^{- \hat c  \xi_0^{\gamma-2} (\hat r - r)^2 }dr, 
\end{split} \end{equation}
where $\hat c = \frac12 c_2 \gamma(1-\gamma)$.
Taking into account \eqref{eq:rstima} and  performing the change of variable $w=(\hat c\xi_0^{\gamma-2})^{1/2} (\hat r - r) $, yields
\begin{equation}\label{ultima stima} 
\begin{split} \int_{\hat r/2}^{\hat r}  r^{\frac{N-1}{2} + a'}e^{- \hat c  \xi_0^{\gamma-2} (\hat r - r)^2 }dr& 
\geq C
 \xi_0^{\frac{N-1}{2} + a'}
\int_{\hat r/2}^{\hat r} e^{- \hat c  \xi_0^{\gamma-2} (\hat r - r)^2 } 
dr
\\
&\ge C\, \xi_0^{\frac{N+1}{2} + a'- \frac \gamma2} \int_0^{\frac{\sqrt{\hat c}}{2} \xi_0^{\gamma/2-1}\hat r} e^{- w^2 } \, dw 
\\&
 \ge C\, \xi_0^{\frac{N+1}{2} + a'- \frac \gamma2} \int_0^1 e^{- w^2 } \, dw \ge C \, \xi_0^{\frac{N+1}{2} + a'- \frac \gamma2},
\end{split} \end{equation}
for $\xi_{0}$ sufficiently large.
Finally, as $u,\,v$ are positive functions and recalling \eqref{eq:stima per below},
\eqref{eq:rhat}, 
\begin{align*} \int_{\R^N} u_\xi v  
&\ge \int_{r_0}^{\xi_0/2} dr \int_{|y|<r} u_\xi v dy 
\ge \xi_0^a e^{-b \xi_0}  \int_{\hat r/2}^{\hat r}  r^{\frac{N-1}{2} + a'}e^{c' r^{\gamma} + c (\xi_0 - r)^{\gamma}} \, dr,
 \end{align*}
thanks to the fact that $1<r_{0}<(1+(c/c')^{1/(1-\gamma)})^{-1}\xi_{0}/2<\xi_{0}/2$.
This together with \eqref{stima 7}and \eqref{ultima stima},    gives the desired estimate from below and completes the proof. 
\end{proof}

We can now give the asymptotic decay of $\eps_{R}$ introduced in 
\eqref{defeps}.

\begin{lemma}\label{est epsilon2}
For every $z\in \Sigma$, let $\mu(Gz) $ be defined in \eqref{eq:muGz}. The following conclusions hold.
\\
(i) If $p$ satisfies \eqref{p>2}, then
\[ 
\eps_R \sim R^{-\frac{N-1}{2}+2\tau_1} e^{- \mu(Gz) \sqrt{V_{\infty}} R},
\]
where $\tau_{1}$ is introduced in \eqref{decay2}.
\\
(ii) If $p=2$ and $\alpha \in (N-1, N-\frac12]$, then
\[ 
\eps_R \sim R^{-\frac{N-1}{2}+\frac\gamma2+2\tau_2} e^{-\mu(Gz) \sqrt{V_{\infty}} R+ 2^{1-\gamma}  c_{\gamma} (\mu(Gz) R)^{\gamma}},
\]
where $\gamma=\alpha-(N-1)$, $c_\gamma$ and $\tau_{2}$ are given in \eqref{decaypertexp}.
\end{lemma}
\begin{remark}
Notice that, for $p=2$ and any $\alpha \in (0, N)$, one can easily give a bound from below on $\eps_R$, which however in general is far from being sharp. One has  
\[
\varepsilon_R^{ij} \ge C R^{-\frac{N-1}{2}} e^{-   d_{i,j} \sqrt{V_{\infty}} R }, \]
see Remark 3.3 in \cite{MaPeScProc} for the the case $d_{i, j}=2$: exactly the same proof also works  for the more general case we are considering here. 

This estimate turns out to be enough in order to consider the case $\ell(G) \ge 3$, and $\beta > \mu_G \sqrt{V_\infty}$, see also \cite{clasal}, as the leading term in the asymptotic 
analysis is the linear part in the exponential. On the other hand, in other cases, and in particular if $p=2$, $\alpha \ge N-1$, exponential and polynomial corrections turn out to be relevant as well, and a more careful analysis is needed. Lemmas \ref{ACR},  \ref{lemma estensione1}, \ref{lemma estensione2} will be crucial. 
\end{remark}
\begin{proof}
Recalling \eqref{defeps} and performing a change of variable one obtains
\[
\eps^{ij}_{R}=\int_{\R^N} (I_{\alpha}\ast\omega^{p})(x)\omega^{p-1}(x)\omega(x-R(g_{j}z-g_{i}z))dx.
\]
We are going to apply Lemma \ref{ACR} with 
\begin{equation}\label{choiceAR}
v=I_{\alpha} \ast \omega^p \omega^{p-1}, \qquad u=\omega, \qquad \xi_{ij}=R(g_i z - g_j z), \; \text{and $|\xi_{ij}|=Rd_{ij}$},
 \end{equation}
 where $d_{ij}$ is introduced in Remark \ref{dij}.
If $p>2$, one takes into account  \eqref{decay2} and  \eqref{decayconv} to deduce that $u$ and $v$  
satisfy the assumptions of Lemma \ref{ACR} with $a=-\frac{N-1}{2}$, $b=\sqrt{V_{\infty}}
$,  $a'=-(p-1)\frac{N-1}{2}-N+\alpha$ and  $b'=(p-1)\sqrt{V_{\infty}}$. 
Since $b < b'$, it results
\[
\eps^{ij}_{R}\sim e^{-\sqrt{V_{\infty}}d_{ij}R}R^{-\frac{N-1}2}, \qquad \text{as $R\to \infty$.}
\]
Then, observing that, by definition,  $\mu(Gz)\leq d_{ij}$ 
and  it is achieved (see Lemma \ref{le:sigmamu} and Remarks \ref{re:muG}, \ref{dij}),  \eqref{defeps} yields the conclusion. 

When $p=2$, it follows that $b=b'$. 
Furthermore, if $\alpha<N-1$, \eqref{decayconv} and 
\eqref{decay2} still hold  so that the conclusion follows as in the case $p>2$.

When  $\alpha=N-1$,  one takes into account
 \eqref{decay2} and obtains   $a=\frac{\nu}
{2}\sqrt{V_{\infty}} -\frac{N-1}{2}$, and $a'=a-1$, so that, $a'<a$ and  
$a'>-\frac{N+1}{2}$,  as  $\nu >0$. Then, the proof of the 
first conclusion is completed observing that 
\( a+a'+\frac{N+1}{2}=   \nu \sqrt{V_{\infty}}-\frac{N-1}{2}\) and applying   Lemma \ref{ACR}.
\\
In order to prove the second conclusion, we perform the same choice as \eqref{choiceAR}. As before $b=b'$, but  Lemma \ref{ACR} cannot be applied,
as it does not include decay such as \eqref{decaypertexp}.
We can instead exploit  Lemma \ref{lemma estensione2} with 
$a=-\frac{N-1}{2}+\tau_2$, $a'= a - N +\alpha$,  $\gamma'=\gamma$, $c=c'=c_\gamma$.
\end{proof}

All the estimates above hold for any $z \in \Sigma$. In order to  compare the asymptotic decay of the potential integral term with $\varepsilon_R$, we need to choose a suitable $z$.
From now on, taking into account Lemma \ref{le:sigmamu}, we fix   
$z \in \Sigma$ such that 
\begin{equation}\label{eq:z}
\mu_G=\mu(Gz), 
\end{equation} 
where $\mu_G$ and $\mu(Gz)$ are defined in \eqref{eq:muGz}, \eqref{defmug}.

\begin{lemma}\label{estimatesV2}
Let $\eps_{R}$ be defined in \eqref{defeps} and $\mu_{G}$ be introduced in 
\eqref{defmug}. Moroever, let $z$ be fixed such that \eqref{eq:z} holds.
Assume  \eqref{p>2} and  \eqref{VgeqN-1} or  \eqref{palpha} and \eqref{VgreatN-1}.
Then it results
\[ \mathcal{A}_V:=\int_{\mathbb{R}^N} \left(V(x) - V_{\infty}\right)\left(\chi_{R, z}\right)^2 \le  o(\varepsilon_R), \qquad \text{as $R\to +\infty$.}
\]
\end{lemma}
\begin{proof} 
Let us first assume that  $p$ satisfies \eqref{p>2}, and $V$ satisfies \eqref{VgeqN-1}.
As in the proof of Lemma \ref{estimatesV} we first observe that
\[ 
\int_{\mathbb{R}^N} \left(V(x) - V_{\infty}\right) \omega_{i, R}^2 \le  C \int_{\mathbb{R}^N}  \abs{x}^\sigma e^{-\beta \abs{x}} \omega_{i, R}^2. 
\]
Take
\[ u=\omega^2, \quad v=  \abs{x}^\sigma e^{-\beta \abs{x}}
\quad \xi_{i}=Rg_i z,\; \text{with }  |\xi_{i}|=R\; \text{ for every $i=1,\dots \ell(G)$.}
 \]
Observe that \eqref{decay2} together with the fact that $\mu(Gz) \le 2$ implies that $u$ satisfies the following upper bound
\[
u \le C e^{-\mu(Gz)\sqrt{V_{\infty}}\abs{x}} \abs{x}^{-N+1}.
\] 
Let us first assume that  $\beta$ given in  \eqref{VgeqN-1}  is such that
$\beta > \mu_G \sqrt{V_{\infty}}= \mu(Gz)\sqrt{V_{\infty}}$, due \eqref{eq:z}. 
We apply Lemma \ref{ACR}  with 
\begin{equation}\label{choicecoeff} 
a=-N+1+2\tau_1, \quad b=\mu(Gz) \sqrt{V_{\infty}}, \qquad a'=\sigma, \quad b'=\beta, \end{equation}
where  $\tau_1$ is given in \eqref{decay2}. Hence, 
 $\mathcal{A}_V$ satisfies 
 \[
\mathcal{A}_V\leq  e^{-\mu(Gz)\sqrt{V_{\infty}}R } R^{-N+1+2\tau_1}
 \]
 and  Lemma \ref{est epsilon2} implies that this is 
 $o(\varepsilon_R)$. 

If $\beta=\mu_G \sqrt{V_{\infty}}=\mu(Gz)\sqrt{V_{\infty}}$, then we have as before 
\eqref{choicecoeff} with $b'=b$ and we apply conclusion (ii) in Lemma \ref{ACR}. Thus, if $\sigma < a=-N+1+2\tau_1$, it holds
\[ \mathcal{A}_V \le
\begin{cases}
e^{-\mu(Gz) \sqrt{V_{\infty}}R } R^{-N+1+2\tau_1+\sigma+ \frac{N+1}{2}} & \text{ if } \sigma >- \frac{N+1}{2}\\
e^{-\mu(Gz) \sqrt{V_{\infty}}R } R^{-N+1+2\tau_1} \log R & \text{ if } \sigma=- \frac{N+1}{2}\\
e^{-\mu(Gz) \sqrt{V_{\infty}}R } R^{-N+1+2\tau_1} & \text{ if } \sigma <- \frac{N+1}{2}.
\end{cases} \]
These estimates  and Lemma \ref{est epsilon2} show that $\mathcal{A}_V=o(\eps_{R})$ when $\sigma <  a$. 
An analogous argument can be performed when $\sigma \ge a$, yielding the first 
conclusion.

Let us now assume that \eqref{palpha} and \eqref{VgreatN-1} hold.
In this case we take
\[ u=\omega^2 ,\quad  v=  \abs{x}^\sigma e^{-\beta \abs{x}+c' \abs{x}^{\gamma'}},
\quad  \xi_{i}=Rg_i z,\; \text{with } |\xi_{i}|=R .
\]
Note that, \eqref{decaypertexp} implies
\begin{align}
\label{eq:u1}
 u &\sim C \abs{x}^{-N+1+2\tau_2} e^{-2\sqrt{V_\infty} \abs{x} + 2c_\gamma \abs{x}^\gamma} 
 \\
\label{eq:u2}
&\le C \abs{x}^{-N+1+2\tau_2} e^{-\mu(Gz) \sqrt{V_\infty} \abs{x} + 2^{1-\gamma} c_\gamma \mu(Gz)^\gamma \abs{x}^\gamma} .
\end{align}
If $\beta > \mu_G \sqrt{V_\infty}= \mu(Gz) \sqrt{V_\infty}$,  
we apply Lemma \ref{lemma estensione1} with
\[ 
a=-N+1+2\tau_2, \; b=\mu(Gz) \sqrt{V_\infty},\; c=2^{1-\gamma} c_\gamma \mu(Gz)^\gamma \quad  a'=\sigma, \; b'=\beta.
\] 
Thus,  ${\mathcal A}_{V}=o(\eps_R)$, taking into account Lemma
\ref{est epsilon2} and recalling that $\gamma>0$.

Let now $\beta=\mu_G\sqrt{V_\infty}= \mu(Gz) \sqrt{V_\infty}$.
The case $\mu_{G}<2$,  is taken into account both in conclusions
(ii) and (iii) of Theorem \ref{mainthm4}, and we will handle it at the same time:
we take into consideration \eqref{eq:u1} and
we  apply Lemma \ref{lemma estensione1}   with
\[ 
a=-N+1+2\tau_2,\; b=2 \sqrt{V_\infty},\;  c=2c_\gamma,
\quad a'=\sigma, \; b'=\beta,   \]
yielding 
\({\mathcal A}_{V} \le C\, R^{\sigma} e^{-\mu_G\sqrt{V_\infty} R + c' R^{\gamma'}}.\)
Then Lemma \ref{est epsilon2} and \eqref{VgreatN-1} yield the conclusion if 
either $\gamma'<\gamma$  or $\gamma'=\gamma$ 
and $c'<2^{1-\gamma}c_{\gamma}\mu_{G}^{\gamma}$ or
$\gamma'=\gamma$ , $c'=2^{1-\gamma}c_{\gamma}\mu_{G}^{\gamma}$
and $\sigma$ satisfies the hypotheses in conclusion (iii) in Theorem \ref{mainthm4}.
In the last case $\beta=\mu_G \sqrt{V_\infty}$ and $\mu_G=2$, 
 \eqref{VgreatN-1} lead us to assume the hypotheses in conclusion (ii) namely
 $\gamma'<\gamma$, then 
 Lemma \ref{lemma estensione1} implies
\[ 
{\mathcal A}_{V} \le C\, R^{-N+1+2\tau_2} e^{-2\sqrt{V_\infty} R + 2c_\gamma R^\gamma}, \]
Then, one deduces that 
${\mathcal A}_{V}=o(\varepsilon_R)$  exploiting Lemma 
\ref{est epsilon2},  using that $2^{1-\gamma}c_{\gamma}(\mu_{G})^{\gamma}=2c_{\gamma}$ as $\mu_{G}=2$ and recalling that $\gamma>0$.
 \end{proof}

\begin{remark}\label{rmk:condizioniottimali}
An inspection of the proofs above provides examples of potentials 
 not satisfying our assumptions and for which the associated 
integral term $\mathcal{A}_V$ is  not $o(\varepsilon_R)$. 

In particular, take $V(x)=V_\infty + \abs{x}^\sigma e^{-\beta \abs{x}+c' \abs{x}^{\gamma'}}$, $\mu_G=2$, 
$\gamma'=\gamma>0$ and $c'>0$. In this case, by Lemma \ref{lemma estensione2}, one gets 
\[\int_{\mathbb{R}^N} (V(x) - V_{\infty}) \omega_{i, R}^2 \sim C\,  R^{\frac{3-N}{2} + \sigma -\frac \gamma 2+2\tau_2} e^{-2 \sqrt{V_\infty}R + \tilde c R^\gamma}, \]
where $\tilde c= \left((c')^{\frac{1}{1-\gamma}} + (2c_\gamma)^{\frac{1}{1-\gamma}} \right)^{1-\gamma}$. 
On the other hand, from Lemma \ref{est epsilon2} we deduce that $\eps_{R}$ decays as follows
\[
\eps_{R}\sim R^{-\frac{N-1}2+\frac{\gamma}2+2\tau_{2}}e^{-\mu(Gz)\sqrt{V_{\infty}}R+2^{1-\gamma}c_{\gamma}(\mu(Gz)R)^{\gamma}}.
\]
Notice that $\mu_G=2$ implies $\mu(Gz)=2$ for any $z \in \Sigma$. Hence we need to take into account the exponential correction and as it holds
$\tilde c > 2 c_\gamma$ for any choice of $c'>0$,  
we deduce that $\mathcal{A}_V$ is not $o(\varepsilon_R)$. 
\end{remark}

\subsection{Proof of Theorem \ref{mainthm2} and \ref{mainthm4}}\label{proofs2} 
We will follow the same strategy of Theorem \ref{mainthm1}. Here, the Nehari manifold is $C^1$, as  $\mathcal{I}_V$ is $C^2$ if $p \ge 2$. Moreover, we point out that conclusions (1) and (2) of Lemma \ref{Nehari homeo} are still true in the setting $p \ge 2$. 

The analog of Proposition \ref{prop:somma} will be the following 
\begin{proposition}\label{prop:somma2}
If $p \ge 2$, then 
\begin{equation}\label{eq:p>21}
 \int_{\mathbb{R}^N}   \left(I_{\alpha} \ast \chi^p_{R, 
z} \right)\chi^p_{R, z} \ge \sum_{i=1}^{\ell(G)} \int_{\R^N} (I_\alpha \ast \omega_{i, R}^p) \omega_{i, R}^p + 2(p-1) \eps_R. 
\end{equation}
For $p=2$ a sharper estimate holds: 
\begin{equation}\label{eq:p>22}
 \int_{\mathbb{R}^N} \left(I_{\alpha} \ast \chi_{R, z}^2\right)\chi_{R, z}^2 
\ge \sum_{i=1}^{\ell(G)}  \int_{\mathbb{R}^N} (I_{\alpha} \ast \omega_{i, R}^2)  \omega_{i, R}^2
+ 4 \varepsilon_R. 
\end{equation}
\end{proposition}
\begin{proof}
The first statement is an immediate consequence of \cite[Lemma 5.3]{clasal}, 
whereas the second one follows by direct computations, see also 
\cite{MaPeScProc}. 
\end{proof}
\begin{remark}
The inequality proved in Proposition \ref{prop:somma2} for $p>2$ is  not 
consistent with the case $p=2$. This is because \eqref{eq:p>21} lies on
an algebraic inequality of Bernoulli's type (see formula (5.2) in \cite{clasal}),
while \eqref{eq:p>22} is obtained by direct computations.
We believe that it would be possible to improve \eqref{eq:p>21}
following the argument of \cite[Lemma 2.2]{acclpa}.
\end{remark}
In order to prove our existence results the following estimate will be crucial. 
\begin{proposition}\label{estimatesIV2}
Let $z$ be fixed in \eqref{eq:z}. Assume \eqref{p>2} and  \eqref{VgeqN-1} or  \eqref{palpha} and \eqref{VgreatN-1}.
Then, the following  inequality holds 
\[ \mathcal{I}_V(T(\chi_{R, z}) \chi_{R, z}) \le 
\begin{cases}
 \ell(G) c^G_{\infty} -\frac{p-2}{2p} \eps_{R}+o(\eps_{R}), & \text{if $p>2$},
 \medskip\\
 \ell(G) c^G_{\infty} -\frac12 \eps_{R}+o(\eps_{R}), & \text{if $p=2$},
\end{cases}
\]
as $R\to +\infty$ and 
where $T(\chi_{R, z}) $ is defined in Lemma \ref{Nehari homeo}.
\end{proposition}
\begin{proof}
Let us first notice that, following Conclusion (1) of Lemma \ref{Nehari homeo}, it is easy to obtain that $T_{R}:=T(\chi_{R, z}) $ is given by
\[
T_{R}^{2p-2}
=
\frac{\|\chi_{R, z}\|^{2}_{V}}{\dys\int_{\mathbb{R}^N}   \left(I_{\alpha} \ast \chi^p_{R, 
z} \right)\chi^p_{R, z}}.
\]
Therefore, taking into account \eqref{normainfty}, \eqref{defeps}, Proposition \ref{prop:somma2}, Lemma \ref{estimatesV2} and that $\omega_{i}$ is a solution of Problem \eqref{Choqlimit} it results
\[\begin{split}
\mathcal{I}_{V}(T_R \chi_{R, z})
&=
\left(\frac12-\frac1{2p}\right)T_{R}^{2}\|\chi_{R,z}\|_{V}^{2}
=
\left(\frac12-\frac1{2p}\right)\frac{(\|\chi_{R,z}\|_{V}^{2})^{\frac{p}{p-1}}}{\left[\dys\int_{\mathbb{R}^N}   \left(I_{\alpha} \ast \chi^p_{R, 
z} \right)\chi^p_{R, z}\right]^{\frac1{p-1}}}
\\
&\leq
\left(\frac12-\frac1{2p}\right)
\dfrac{\left[\dys \sum_{i=1}^{\ell(G)}\|\omega_{i,R}\|^{2}+\eps_{R}+o(\eps_{R})
\right]^{\frac{p}{p-1}}}{
\left[\dys \sum_{i=1}^{\ell(G)}\|\omega_{i,R}\|^{2}
+b_p \eps_{R}\right]^{\frac{1}{p-1}}},
\end{split}
\]
where
\[ b_p=\begin{cases}
2(p-1) &\text{ if } p>2 \\
4 &\text{ if } p=2.
\end{cases} \]
Notice that $b_p >p$ for any $p \ge 2$. 
Using the expansion $(a+t)^{\alpha}=a^{\alpha}+\alpha a^{\alpha-1}t+o(t)$  and the notation $a:= \sum  \|\omega_{i,R}\|^{2}$, we get
\[\begin{split}
\mathcal{I}_{V}(T_R \chi_{R, z})
&
\le \left(\frac12-\frac1{2p}\right)
\left[a+\eps_{R}+o(\eps_{R})\right]^{\frac{p}{p-1}}
\left[a+b_p\eps_{R}\right]^{-\frac{1}{p-1}}
\\&
=\left(\frac12-\frac1{2p}\right)\left[
a^{\frac{p}{p-1}}+\frac{p}{p-1}a^{\frac1{p-1}}\eps_{R}+o(\eps_{R})\right]
\left[a^{-\frac{1}{p-1}}-\frac{b_p}{p-1}a^{-\frac{p}{p-1}}\eps_{R}+o(\eps_{R})\right]
\\
&=
\left(\frac12-\frac1{2p}\right)\left[
a- \eps_{R}\frac{b_p-p}{p-1} +o(\eps_{R})\right]
=\ell(G)c^G_{\infty}-\frac{b_p -p}{2p}\eps_{R}+o(\eps_{R}),  
\end{split}
\]
where the last equality comes from the fact that $c^{G}_{\infty}=\left(\frac12-\frac1{2p}
\right) \|\omega_{i,R}\|^{2}$.
\end{proof}

We now prove our main results in case $p \ge 2$. 
\begin{proof}[Proof of Theorems \ref{mainthm2} and \ref{mainthm4}]
 Let us take a minimizing sequence for 
$c_V^G$ and exploit  Ekeland's Variational Principle \cite{eke} to construct a 
minimizing sequence which is  also a  Palais-Smale for ${\mathcal I}_{V}$ restricted on 
$\mathcal{N}^G_V$, then arguing as in Corollary 3.2 in \cite{clasal1} 
(see also Lemma 2.2 and Lemma 2.5 in 
\cite{ClappMaia}) we obtain a subsequence $u_{n}$ which is a Palais-Smale 
sequence in the whole $H^{1}_{G}$. 

Take $z$ satisfying \eqref{eq:z},  exploit Proposition \ref{estimatesIV2} to
 apply Proposition 3.1 in 
\cite{CingolaniClappSecchi} and deduce that $u_{n}$  is compact.
Then, there exists $u \in \mathcal{N}^G_V$ 
such that $\mathcal{I}_V(u)=c_V^G$. Also $\abs{u} \in \mathcal{N}^G_V$ and
$c_{V}=\mathcal{I}_{V}(u)=\mathcal{I}_{V}(|u|)$ so that we can take $u$ positive. 

Hence, we have a $G$-invariant positive solution. 
\end{proof}
\appendix \section{Technical Lemma }\label{app:lemma}

\begin{lemma}[Lemma 4.1 in \cite{ClappMaia2}]\label{lemma cm}
Let $u, v : \mathbb{R}^N \to \mathbb{R}$ be two continuous functions such that
\[ u(x) \le C(1+\abs{x})^a, \quad v(x) \le C(1+\abs{x})^{a'}  \]
as $\abs{x} \to \infty$, where $a, a'<0$ such that $a+a'<-N$. Let $\xi \in \mathbb{R}^N $ such that $\abs{\xi} \to \infty$. We denote $u_{\xi}(x)=u(x-\xi)$. Then the following asymptotic estimate holds:
\[ \int_{\mathbb{R}^N} u_{\xi} v\le C \abs{\xi}^{\tau} \]
where $\tau = \max \{ a, a', a+a' +N\}<0$. 
\end{lemma}

\begin{lemma}[Lemma 3.7 in \cite{AmbrosettiColoradoRuiz}]\label{ACR}
Let $u, v : \mathbb{R}^N \to \mathbb{R}$ be two positive continuous radial functions such that
\[ u(x) \sim \abs{x}^a e^{-b\abs{x}}, \quad v(x) \sim \abs{x}^{a'} e^{-b'\abs{x}}  \]
as $\abs{x} \to \infty$, where $a, a' \in \mathbb{R}$, and $b, b' >0$. Let $\xi \in \mathbb{R}^N $ such that $\abs{\xi} \to \infty$. We denote $u_{\xi}(x)=u(x-\xi)$. Then the following asymptotic estimates hold:
\begin{itemize}
\item[(i)] If $b < b'$,
\[ \int_{\mathbb{R}^N} u_{\xi} v\sim e^{-b\abs{\xi}} \abs{\xi}^{a}. \]
A similar expression holds if $b > b'$, by replacing $a$ and $b$ with $a'$ and $b'$. 
\item[(ii)] If $b=b'$, suppose that $a \ge a'$. Then:
\[ \int_{\mathbb{R}^N} u_{\xi} v\sim
\begin{cases}
e^{-b\abs{\xi}} \abs{\xi}^{a+a'+\frac{N+1}{2}} & \text{ if } a' > -\frac{N+1}{2},\\
e^{-b\abs{\xi}} \abs{\xi}^{a} \log |\xi| & \text{ if } a' = -\frac{N+1}{2},\\
e^{-b\abs{\xi}} \abs{\xi}^{a}& \text{ if } a' < -\frac{N+1}{2}.
\end{cases} \]
\end{itemize}
\end{lemma}

\section*{Acknowledgments}
The authors wish to thank M\'onica Clapp for inspiring conversations.

\end{document}